\documentclass[11pt]{amsart}
\usepackage{amsaddr}
\usepackage{amssymb,amsfonts,url,color}
\usepackage[all,arc]{xy}
\usepackage{enumerate}
\usepackage{mathrsfs}
\usepackage[utf8]{inputenc}
\usepackage{pdfsync}
\usepackage[margin=0.9in]{geometry}
\usepackage{caption}
\usepackage{graphicx}
\usepackage{hyperref}
\usepackage[makeroom]{cancel}
\usepackage[section]{placeins}
\usepackage{float}
\newtheorem{thm}{Theorem}[section]

\newtheorem{prop}[thm]{Proposition}
\newtheorem{lem}[thm]{Lemma}

\newtheorem{prob}[thm]{Problem}
\newtheorem{rem}[thm]{Remark}

\theoremstyle{definition}

\makeatletter
\let\c@equation\c@thm
\makeatother
\numberwithin{equation}{section}

\title[KdV Equation]{Pseudo-backstepping and its application to the control of Korteweg-de Vries equation from the right endpoint on a finite domain}

\author{Türker Özsarı\textsuperscript{*} \& Ahmet Batal}
\address{Department of Mathematics, Izmir Institute of Technology, Urla, Izmir, TURKEY}
\thanks{\textsuperscript{*}Correspondence: Türker Özsarı, Department of Mathematics, Izmir Institute of Technology, Urla, Izmir 35430, TURKEY; E-mail: turkerozsari@iyte.edu.tr}
\thanks{This research was funded by IZTECH BAP Grant 2017IYTE14}
\begin{document}

\begin{abstract}
In this paper, we design Dirichlet-Neumann boundary feedback controllers for the Korteweg-de Vries (KdV) equation that act at the right endpoint of the domain.  The length of the domain is allowed to be critical.  Constructing backstepping controllers that act at the right endpoint of the domain is more challenging than its left endpoint counterpart. The standard application of the backstepping method fails, because corresponding kernel models become overdetermined.  In order to deal with this difficulty, we introduce the \emph{pseudo-backstepping} method, which uses a \emph{pseudo-kernel} that satisfies all but one desirable boundary condition.  Moreover, various norms of the pseudo-kernel can be controlled through a parameter in one of its boundary conditions.  We prove that the boundary controllers constructed via this pseudo-kernel still exponentially stabilize the system with the cost of a low exponential rate of decay.  We show that a single Dirichlet controller is sufficient for exponential stabilization with a slower rate of decay.  We also consider a second order feedback law acting at the right Dirichlet boundary condition.  We show that this approach works if the main equation includes only the third order term, while the same problem remains open if the main equation involves the first order and/or the nonlinear term(s). At the end of the paper, we give numerical simulations to illustrate the main result.
\end{abstract}

\def\uppercasenonmath#1{} 
\maketitle

\section{Introduction}
This article is devoted to the study of the boundary feedback stabilization of the Korteweg-de Vries (KdV) equation on a bounded domain $\Omega=(0,L)\subset \mathbb{R}$.  The linear version of the model under consideration is given by
\begin{equation}\label{KdVBurgers}
 	\begin{cases}
 	\displaystyle u_{t} + u_{x} + u_{xxx} =0 & \text { in } \Omega\times \mathbb{R_+},\\
    u(0,t) = 0, u(L,t) = U(t), u_{x}(L,t)=V(t)  & \text { in } \mathbb{R_+},\\
    u(x,0)=u_0(x) & \text { in } \Omega,
 	\end{cases}
\end{equation} whereas the nonlinear version of this model is written with the main equation in \eqref{KdVBurgers} replaced with \begin{equation}\label{nonlinearKdV}u_t+u_x+u_{xxx}+uu_x=0.\end{equation}

In \eqref{nonlinearKdV}, $u=u(x,t)$ is a real valued function that can model the evolution of the amplitude of a weakly nonlinear shallow dispersive wave in space and time \cite{KdVPaper}.  The inputs $U(t)=U(u(t,\cdot))$ and $V(t)=V(u(t,\cdot))$ at the right endpoint of the boundary are feedbacks. The goal is to choose these boundary feedbacks so that the solutions of \eqref{KdVBurgers} and \eqref{nonlinearKdV} decay to zero as $t\rightarrow \infty$, at an exponential rate in the mean-square sense.

Controlling the behavior of solutions of evolution equations is an important topic, and many approaches have been proposed.  One method is to use local or global interior controllers.  Another method is to use external (boundary) controllers, especially in those models where it is difficult to access the domain.  Using feedback type controls is a common tactic to stabilize the solutions. However, non-feedback type controls (open loop control systems) are also used for steering solutions to or near a desired state.  Exact, null, or approximate controllability models have been developed for almost all well-known PDEs.

Exact boundary controllability of linear and nonlinear KdV equations with the same type of boundary conditions as in \eqref{KdVBurgers} was studied by \cite{Rosier1997}, \cite{Cor2004}, \cite{Zhang99}, \cite{Glass08}, \cite{Cerpa07}, \cite{Cerpa09}, \cite{RosZha09}, and \cite{Glass10}. In these papers, the boundary inputs are chosen in advance to steer solutions to a desired final state at a given time.  This results in an open loop model. In contrast, the boundary inputs in our model depend on the solution itself, and \eqref{KdVBurgers} is therefore closed loop.

Stabilization of solutions of the KdV equation with a localised interior damping was achieved by \cite{Perla2002}, \cite{Pazo05}, \cite{Mass07}, and \cite{Balogh2000}. There are also some results achieving stabilization of the KdV equation by using predetermined local boundary feedbacks; see for instance \cite{Liu2002}, and \cite{Jia2016}.

\subsection{Motivation}
\eqref{KdVBurgers} and \eqref{nonlinearKdV} with homogeneous boundary conditions ($U=V\equiv 0$) are both dissipative since $\frac{d}{dt}\|u(t)\|_{L^2(\Omega)}^2\le 0.$ However, this does not always guarantee exponential decay.  It is well-known that if $\displaystyle L\in \mathcal{N}\equiv \left\{2\pi\sqrt{\frac{k^2+kl+l^2}{3}}, k,l\in \mathbb{N}\right\}$ (so called \emph{critical lengths} for KdV), then the solution does not need to decay to zero at all.  For example, if $L=2\pi$, $u=1-\cos(x)$ is a (time independent) solution of \eqref{KdVBurgers} on $\Omega=(0,2\pi)$, but its $L^2-$norm is constant in $t$.  On the other hand, if $L$ is not critical, one can show the exponential stabilization of solutions for \eqref{KdVBurgers} under homogeneous boundary conditions; see for example \cite[Theorem 2.1]{Perla2002}.

Recently, \cite{Cerpa2013} studied the boundary feedback stabilization of the KdV equation with the boundary conditions
\begin{equation}\label{BCforKdVLeft}
  u(0,t) = U(t), u_{x}(L,t) = 0, u(L,t) = 0
\end{equation} by using the back-stepping technique (see for example \cite{KrsBook}). \cite{Cerpa2013} proved that given any $r>0$, there corresponds a smooth kernel $k=k(x,y)$ such that the boundary feedback controller $U(t)=U(u(t,\cdot))=\int_0^Lk(0,y)u(y,t)dy$ steers the solution of the linear KdV equation to zero with the decay rate estimate $\|u(t)\|_{L^2(\Omega)}\lesssim \|u_0\|_{L^2(\Omega)}e^{-r t}.$ Moreover, the same result also holds true for the nonlinear KdV equation provided that $u_0$ is sufficiently small in the $L^2-$sense.  Here, $k=k(x,y)$ is an appropriately chosen kernel function satisfying a third order PDE model on a triangular domain that involves three boundary conditions. In \cite{Cerpa2013}, the control acts on the Dirichlet boundary condition at the \emph{left} endpoint of the domain.  However, the situation is very different if the control acts at the \emph{right} endpoint of the domain, because then the kernel of the backstepping controller has to satisfy an overdetermined PDE model whose solution may or may not exist.  Therefore, the problem of finding backstepping controllers acting at the \emph{right} endpoint of the domain is interesting.

Coron \& Lü \cite{Cor14} studied this problem with a single controller acting from the Neumann boundary condition on domains of \emph{uncritical} lengths.  They prove the rapid exponential stabilization of solutions for the KdV equation under a smallness assumption on the initial datum.  The method of \cite{Cor14} is based on using a rough kernel function in the backstepping integral transformation.  The construction of the rough kernel relies on the exact controllability of the linear KdV equation by the Neumann boundary control acting at the right endpoint of the domain.  However, the exact controllability was proved only for the domains of uncritical lengths.  On the other hand, the exponential decay of solutions for the linearized KdV equation holds even without adding any control to the system when the length of the domain does not belong to the set of critical lengths \cite{Perla2002}.  Therefore, the following remains as an important problem:
\begin{prob}\label{mainprob} Let $L>0$ (not necessarily uncritical).  Can you find a kernel $k=k(x,y)$ such that the solution of \eqref{KdVBurgers} and \eqref{nonlinearKdV} satisfies \begin{equation}\label{decayest}\|u(\cdot,t)\|_{L^2(\Omega)}=\mathcal{O}(e^{-rt})\end{equation} for some $r>0$ with boundary feedback controllers given by  \begin{equation}\label{controller}
                          U(t)=\int_0^Lk(L,y)u(y,t)dy \text{ and } V(t)=\int_0^Lk_x(L,y)u(y,t)dy\,?
                        \end{equation}
\end{prob}
A stronger version of the above problem is the following:
\begin{prob}\label{mainprob1}
  Given $r>0$, can you find a kernel $k=k(x,y)$ such that the solution of \eqref{KdVBurgers} and \eqref{nonlinearKdV} satisfies the $L^2-$decay estimate \eqref{decayest} with the boundary feedback controllers given in \eqref{controller}?
\end{prob}

This paper and the method proposed address only Problem \ref{mainprob}, and the latter problem still remains open for domains of critical length.

In order to understand the nature of the problem and the difficulty here, let us consider the linearised KdV equation in \eqref{KdVBurgers}.  A backstepping controller for this linear model is generally constructed by using a transformation given by \begin{equation}\label{transform}w(x,t)\equiv u(x,t)-\int_0^xk(x,y)u(y,t)dy,\end{equation} where the unknown kernel function $k(x,y)$ is chosen in such a way that if $u$ is a solution of \eqref{KdVBurgers} with the boundary feedback controllers given in \eqref{controller}, then $w$ is a solution of the damped homogeneous initial-boundary value problem (so called ``\emph{target system}'')
\begin{equation}\label{HomKdVBurgers}
 	\begin{cases}
 	\displaystyle w_{t} + w_{x} + w_{xxx} + \lambda w = 0 & \text { in } \Omega\times \mathbb{R_+},\\
    w(0,t) = w(L,t) = w_{x}(L,t) = 0 & \text { in } \mathbb{R_+},\\
    w(x,0)=w_0(x)\equiv u_0-\int_0^xk(x,y)u_0(y)dy & \text { in } \Omega.
 	\end{cases}
\end{equation}  The reason is that the solution of \eqref{HomKdVBurgers} satisfies $\|w(t)\|_{L^2(\Omega)}=O(e^{-\lambda t})$, and if the given transformation is invertible, one can hope to get a similar decay property for $u$.

The essence of the back-stepping algorithm is to find an appropriate kernel function $k$ which serves the purpose.  In order to do this, one simply assumes that $u$ solves \eqref{KdVBurgers} and plugs in $u(x,t)-\int_0^xk(x,y)u(y,t)dy$ into the main equation in \eqref{HomKdVBurgers} wherever one sees $w$.  This gives a set of sufficient conditions that the kernel has to satisfy. Note that $w$ satisfies the given homogeneous boundary conditions $w(0,t)=w(L,t)=w_x(L,t)=0$ by the transformation in \eqref{transform} and the choice of the feedback controllers in \eqref{controller}.  In order for the main equation in \eqref{HomKdVBurgers} to be satisfied, one can impose a few conditions on $k$.  Indeed, computing the derivative of $w$ with respect to the temporal and spatial derivatives and putting these together, we obtain the following:
\begin{gather}
  \nonumber w_{t}(x,t) + w_{x}(x,t) + w_{xxx}(x,t) + \lambda w(x,t)= k_y(x,0)u_x(0,t)\\
  -\int_{0}^{x}u(y,t)\left[k_{xxx}(x,y) + k_{x}(x,y) + k_{yyy}(x,y) + k_{y}(x,y) + \lambda k(x,y)\right]dy \label{23} \\
  \nonumber -k(x,0)u_{xx}(0,t) - u_{x}(x,t)\left[k_{y}(x,x) + k_{x}(x,x) + 2\frac{d}{dx}k(x,x)\right] \\
  \nonumber + u(x,t)\left[\lambda - k_{xx}(x,x) + k_{yy}(x,x) - \frac{d}{dx}k_{x}(x,x) - \frac{d^{2}}{dx^{2}}k(x,x)\right].
\end{gather}
The above equation is the same as that of the target system \eqref{HomKdVBurgers} if $k$ solves the third order partial differential equation together with the set of boundary conditions given by
\begin{eqnarray} \label{kEq}
  \nonumber k_{xxx} + k_{yyy} + k_{y} + k_{x}   &=& -\lambda k, \\
   k(x,x) = k(x,0)=k_y(x,0) &=& 0, \\
  \nonumber k_{x}(x,x) &=& \frac{\lambda}{3}x,
\end{eqnarray} where the PDE model is considered on the triangular spatial domain $\mathcal{T}\equiv \{(x,y)\in \mathbb{R}^2\,|\,x\in [0,L], y\in [0,x]\}\,\,\text{(see Figure \ref{regT} below)}.$
\begin{figure}[H]
  \centering
   \includegraphics[scale=0.50]{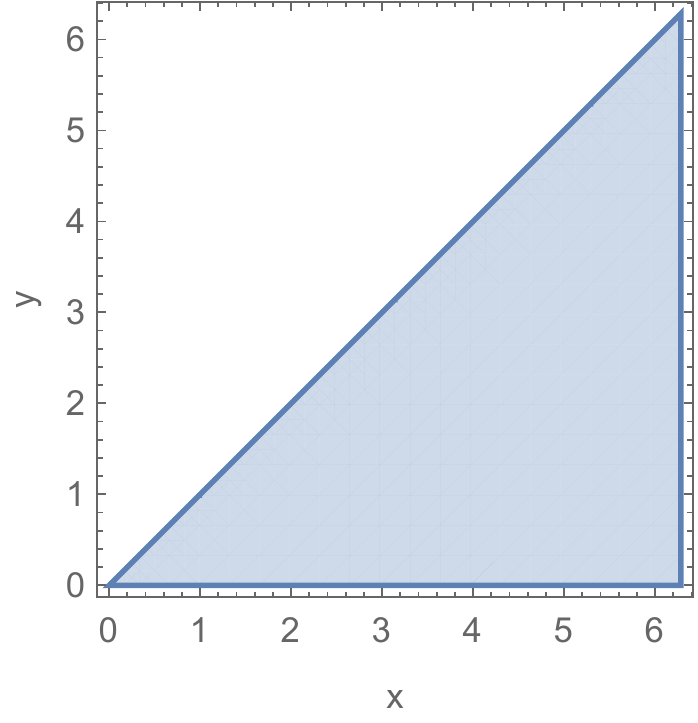}
  \caption{Triangular region $T$ for $L=2\pi$}\label{regT}
\end{figure}
In order to solve the problem \eqref{kEq}, one generally first applies a change of variables. Here, an appropriate choice would be to define $t\equiv y$, $s\equiv x-y$, and $G(s,t)\equiv k(x,y)$.  Then, $G$ satisfies the boundary value problem given by
\begin{eqnarray}
  \label{Geq}G_{ttt} - 3G_{stt}+ 3G_{sst} + G_{t} &=& -\lambda G, \\
  \label{Geqb}G(s,0) =  G_t(s,0) = G(0,t) &=& 0, \\
  \label{Geqc}G_s(0,t) &=& \frac{\lambda}{3}t
\end{eqnarray} on the triangular domain $\mathcal{T}_{0}\equiv \left\{ (s,t) \,|\, t \in [0,L], s \in [0,L-t]\right\}\,\,\text{(see Figure \ref{regT0} below)}.$

Unfortunately, it is not easy to decide whether \eqref{Geq}-\eqref{Geqc} has a solution.  Note that there is also a mismatch between the boundary conditions $ G_t(s,0)=0$ and $G_s(0,t) = \displaystyle \frac{\lambda}{3}t$ in the sense that $G_{ts}(0,0)=0\neq G_{st}(0,0)=\displaystyle \frac{\lambda}{3}.$  Hence, the standard back-stepping algorithm fails because it enforces us to solve an overdetermined singular PDE model.  This issue does not arise if one tries to control the system from the left endpoint of the domain as in \cite{Cerpa2013}.

\begin{figure}[H]
  \centering
   \includegraphics[scale=0.50]{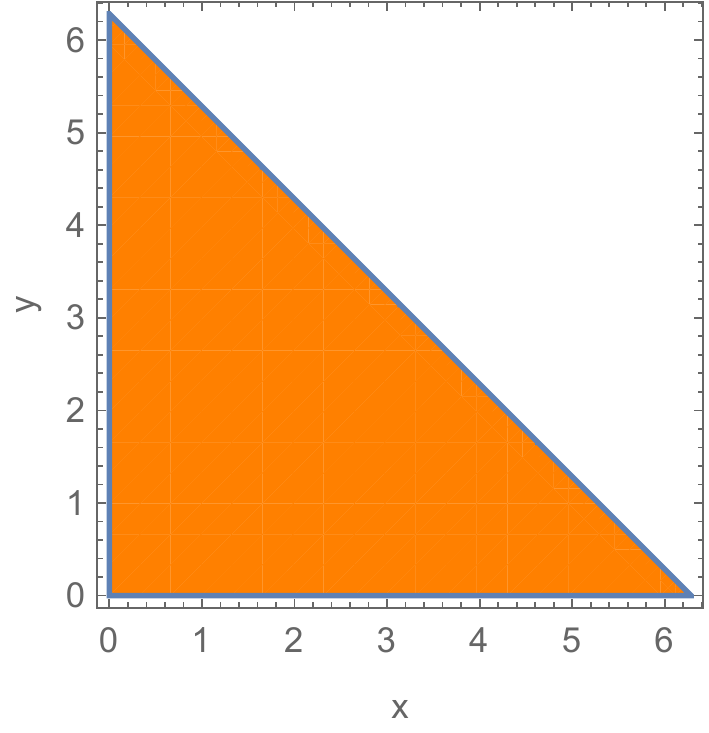}
  \caption{Triangular region $T_0$ for $L=2\pi$}\label{regT0}
\end{figure}

The adverse effect of the nonhomogeneous boundary condition in the kernel PDE model was eliminated by expanding the domain from a triangle into a rectangle in \cite{Cor14}.  However, this approach brings a dirac delta term to the right hand side of the main equation; see the kernel model in \cite[Section 1]{Cor14}.  The cost of this is that the constructed kernel cannot be expected to be very smooth.  However, the higher regularity is crucial to rigorously justify the calculations in \eqref{23} that show the equivalence of the original plant and the exponentially stable target system.  Therefore, we rely on a different idea based on constructing an imperfect but smooth kernel. The details of this construction are given below.

\subsection{Pseudo-backstepping} We introduce a new backstepping technique which eliminates the difficulties explained in the previous section.  In the standard backstepping method, the plant model \eqref{KdVBurgers} is transformed into the most desirable (e.g., exponentially stable) target system with a transformation as in \eqref{transform}.  This is called forward transformation.  The target system is then transformed back into the plant model via an inverse transformation, generally in the form \begin{equation}\label{backwardt}u(x,t)=w(x,t)+\int_0^xp(x,y)w(y,t)dy.\end{equation}  This is called backward transformation.  A combination of these two steps allows one to conclude that the plant is stable if and only if the target system is stable in the same sense (see Figure \ref{Backstepping}).
\begin{figure}[H]
  \centering
   \includegraphics[scale=0.75]{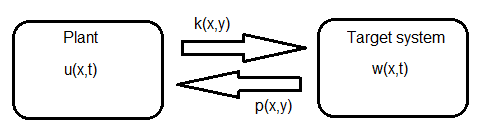}
  \caption{Standard back-stepping}\label{Backstepping}
\end{figure}
Unfortunately, applying this algorithm to our problem forces kernels $p$ and $k$ to be solutions of overdetermined boundary value problems, and thus the method fails.

Our strategy uses a pseudo-kernel which is chosen as a solution of a corrected version of the gain control PDE given by:
\begin{eqnarray}
  \label{Geqeps1}\tilde{G}_{ttt} - 3\tilde{G}_{stt}+ 3\tilde{G}_{sst} + \tilde{G}_{t} &=& -\lambda \tilde{G}, \\
  \label{Geqeps2}\tilde{G}(s,0) =   \tilde{G}(0,t) &=& 0, \\
  \label{Geqeps3}\tilde{G}_s(0,t) &=& \frac{{\lambda}}{3}t
\end{eqnarray} on the triangular domain $\mathcal{T}_{0}$.  Unlike in the previous model \eqref{Geq}-\eqref{Geqc}, here the boundary condition $\tilde{G}_t(s,0)=0$ is completely disregarded.  One advantage of using this modified model is that we can solve it. Another is that, even though the boundary condition $\tilde{G}_t(s,0)=0$ is disregarded, we can control the size of this boundary condition by choosing ${\lambda}$ sufficiently small.  The cost of using a pseudo-kernel is that the target system changes, (see the modified target system in \eqref{HomKdVBurgers-1}), which causes a slower rate of decay.    Nevertheless, this new method (henceforth referred to as \emph{pseudo-backstepping}) allows us to obtain physically reasonable exponential decay rates for some choice of $\lambda$ (see Table \ref{table1} for sample decay rates for some values of $\lambda$ on a domain of length $L=2\pi$).

Another aspect of our method is that instead of using a concrete backward transformation as in \eqref{backwardt}, we rely on the existence of an abstract inverse transformation that maps the solution of the modified target system back into the original plant.   The existence of such a transformation is proved via succession (see Lemma \ref{inverselem} below).  This type of backward transformation was previously used in the stabilization of the heat equation with a localized source of instability \cite{Liu03}.  We do not search for an inverse of type \eqref{backwardt} to avoid a highly overdetermined system that would result from computing the temporal and spatial derivatives of the given transformation and finding the conditions that $p$ has to satisfy.
\begin{figure}[h]
  \centering
   \includegraphics[scale=0.75]{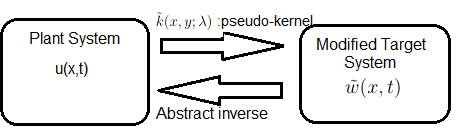}
  \caption{Pseudo-backstepping}\label{pseudo-bs}
\end{figure}

\subsection{Main results}

Applying the pseudo-backstepping method explained above to the linearized and nonlinear KdV models given in \eqref{KdVBurgers} and \eqref{nonlinearKdV}, we are able to prove the following wellposedness and stabilization theorems:
\begin{thm}[Wellposedness]\label{Linthm0}
Let $T>0$, $u_0\in L^2(\Omega)$ and \begin{equation}\label{controllers}U(t) = \int_0^L\tilde{k}(L,y)u(y,t)dy,\,V(t)= \int_0^L\tilde{k}_x(L,y)u(y,t)dy,\end{equation} where $\tilde{k}$ is a smooth kernel given by \eqref{ktilde}.
Then, \eqref{KdVBurgers} has a unique solution $u\in C([0,T];L^2(\Omega))\cap L^2(0,T;H^1(\Omega))$ satisfying also $u_x\in C([0,L];L^2(0,T)).$  Moreover, the same result also holds true for the nonlinear KdV equation \eqref{nonlinearKdV} if $\|u_0\|_{L^2(\Omega)}$ is sufficiently small.
\end{thm}

\begin{rem} Indeed, our analysis in this paper also shows that if $u_0\in H^3(\Omega)$ and satisfies the compatibility conditions \begin{equation}\label{compcond}u_0(0)=0, u_0(L)=\int_0^L\tilde{k}(L,y)u_0(y)dy, u_0'(L)=\int_0^L\tilde{k}_x(L,y)u_0(y)dy,\end{equation} then the solution of \eqref{KdVBurgers} or the local solution of \eqref{nonlinearKdV} satisfies $u\in C([0,T];H^3(\Omega))\cap L^2(0,T;H^4(\Omega)).$
  One can also interpolate to get regularity in the fractional spaces.  For example, let $u_0\in H^s(\Omega)$ ($s\in [0,3]$) so that it satisfies the compatibility conditions $u_0(0)=0, u_0(L)=\int_0^L\tilde{k}(L,y)u_0(y)dy$ if $s\in [0,3/2]$ and the compatibility conditions \eqref{compcond} if $s\in (3/2,3].$ Then, the solution of \eqref{KdVBurgers} or the local solution of \eqref{nonlinearKdV} satisfies $u\in C([0,T];H^s(\Omega))\cap L^2(0,T;H^{s+1}(\Omega)).$
\end{rem}

\begin{thm}[Stabilization]\label{Linthm} Let $u_0\in L^2(\Omega)$.  Then, for sufficiently small $\lambda>0$, one has $\alpha=\lambda-\frac{1}{2}\|\tilde{k}_y(\cdot,0)\|_{L^2(\Omega)}^2>0$, where $\tilde{k}$ is given by \eqref{ktilde}, and the corresponding solution of \eqref{KdVBurgers} with the boundary feedback controllers \eqref{controllers} satisfies
  $ \|u(t)\|_{L^2(\Omega)}\lesssim \|u_0\|_{L^2(\Omega)}e^{-\alpha t}.$ Moreover, the same decay property is also true for the nonlinear KdV equation \eqref{nonlinearKdV} if $\|u_0\|_{L^2(\Omega)}$ is sufficiently small.
\end{thm}

\begin{rem}
The proof of Theorem \ref{Linthm} is given in the next section.  Table \ref{table1} gives some examples where exponential stabilization can be achieved.  For example, when $\lambda=0.03$, the decay rate is approximately of order $\mathcal{O}(e^{-0.18 t})$ on a domain of length $L=2\pi$.  The exponential decay rate is substantially small (see Table \ref{table1}) relative to the decay rates one can get by controlling the equation from the left end-point with the same type of boundary conditions. Indeed, what matters is is not where the controller is located,but rather the number of boundary conditions specified on the opposite side of the boundary.  For example, if one specified two boundary conditions at the left and only one boundary condition at the right, then it would be easier to control from right and more difficult to control from left, in contrast to the problem studied in this paper.
\end{rem}
\begin{table}
  \begin{tabular}{|c||c|}
  \hline
  $\lambda$ &  $\alpha=\lambda-\frac{1}{2}\|\tilde{k}_y(\cdot,0)\|_{L^2(\Omega)}^2$  \\\hline\hline
   0.01 &  \textcolor[rgb]{0.00,0.50,0.00}{0.00954938} \\\hline
   0.02 &  \textcolor[rgb]{0.00,0.50,0.00}{0.0167563} \\\hline
   0.03 &  \textcolor[rgb]{0.00,0.50,0.00}{0.0181985} \\\hline
   0.04 &  \textcolor[rgb]{0.00,0.50,0.00}{0.00844268} \\\hline
   0.05 &   \textcolor[rgb]{1.00,0.00,0.00}{-0.0203987} \\\hline
   0.10 & \textcolor[rgb]{1.00,0.00,0.00}{-0.961935} \\\hline
   1 &   \textcolor[rgb]{1.00,0.00,0.00}{-83925.8} \\\hline
  \hline
\end{tabular}
\caption{\label{table1}Numerical experiments on a domain of critical length $L=2\pi$}
\end{table}
%
\section{Stabilization} In this section, we prove Theorem \ref{Linthm}.  First, we prove the existence of the pseudo-kernel and the abstract inverse transformation.  Secondly, by using the multiplier method, we obtain the stabilization for suitable $\lambda.$  The multiplier method is applied only formally, but the calculations can be justified by a standard density argument and the regularity results proved in the next section.
\subsection{Linearised model}\label{ProofofThm}
The sought-after solution of \eqref{Geqeps1}-\eqref{Geqeps3} can be constructed by applying the successive approximations technique to the integral equation
\begin{equation}\label{GepsInt}
  \tilde{G}(s,t) = \frac{{\lambda}}{3}st+\frac{1}{3}\int_0^t\int_0^s\int_0^\omega (-\tilde{G}_{ttt} + 3\tilde{G}_{stt}  - \tilde{G}_{t} -\lambda \tilde{G})(\xi,\eta)d\xi d\omega d\eta.
\end{equation}
Indeed, we have the following lemma.
\begin{lem}\label{lemback} There exists a $C^\infty$-function $\tilde{G}$ such that $\tilde{G}$ solves the integral equation \eqref{GepsInt} as well as the boundary value problem given in \eqref{Geqeps1}-\eqref{Geqeps3}.
\end{lem}
\begin{proof}Let $P$ be defined by
\begin{equation}\label{aP}
  (P f)(s,t) = \frac{1}{3}\int_0^t\int_0^s\int_0^\omega (-f_{ttt} + 3f_{stt}  - f_{t} -\lambda f)(\xi,\eta)d\xi d\omega d\eta.
\end{equation}
By \eqref{GepsInt}, we need to solve the equation $\tilde{G}(s,t)=\frac{{\lambda}}{3}st+P\tilde{G}(s,t).$ Define $\tilde{G}^0\equiv 0,$ $\displaystyle\tilde{G}^1(s,t)=\frac{{\lambda}}{3}st,$ and $\tilde{G}^{n+1}=\tilde{G}^1+P \tilde{G}^n.$ Then for $n\geq 1$, $\tilde{G}^{n+1}-\tilde{G}^{n} = P(\tilde{G}^{n}-\tilde{G}^{n-1}).$ So if we define $H^0(s,t)=st$ and $H^{n+1}=PH^n$, we get $H^n=\frac{3}{\lambda}(\tilde{G}^{n+1}-\tilde{G}^{n}).$ Moreover, for $j>i,$
\begin{equation}\label{aCauchy}
\tilde{G}^j-\tilde{G}^i= \sum_{n=i}^{n=j-1}\tilde{G}^{n+1}-\tilde{G}^{n}=\frac{\lambda}{3}\sum_{n=i}^{n=j-1}H^{n}.
\end{equation}
Let $\| \cdot \|_{\infty}$ denote the supremum norm of a function on the triangle $T_0$. It follows from \eqref{aCauchy} that to prove $\tilde{G}_n$ (and its partial derivatives) is Cauchy with respect to the norm $\| \cdot \|_{\infty}$ it is enough to show $H^n$ (and its partial derivatives) is an absolutely summable sequence with respect to the same norm.

To show $H^n$'s are absolutely summable, let us first write $P$ as the sum of four operators $P= P_{-2}+P_{-1}+P_0+P_1,$ where
$$P_{-2}f= \frac{1}{3}\int_0^t\int_0^s\int_0^\omega -f_{ttt}(\xi,\eta) d\xi d\omega'd\eta,\,P_{-1}f= \int_0^t\int_0^s\int_0^\omega f_{stt}(\xi,\eta) d\xi d\omega'd\eta,$$
$$P_{0}f= \frac{1}{3}\int_0^t\int_0^s\int_0^\omega -f_{t}(\xi,\eta) d\xi d\omega'd\eta,\,P_{1}f= \frac{1}{3}\int_0^t\int_0^s\int_0^\omega -\lambda f(\xi,\eta) d\xi d\omega'd\eta.$$
Then
\begin{equation}\label{aproduct}
H^n=P^nH^0=(P_{-2}+P_{-1}+ P_0+P_1)^nst=\sum_{r=1}^{4^n}R_{r,n}st
\end{equation}
where $R_{r,n}:=P_{j_{r,n}}P_{j_{r,n-1}}\cdot\cdot\cdot P_{j_{r,1}}$, $j_{r,i} \in \{-2,-1,0,1\}$.  Observe that for positive integers $m$ and nonnegative integers $k$
\begin{equation}\label{aPi}
P_{-1}s^m t^k= c_{-1}s^{m+1}t^{k-1} \; \text{and} \; P_{i}s^m t^k= c_{i}s^{m+2} t^{k+i} \; \text{for} \; i=-2,0,1,
\end{equation}
where
\begin{equation}\label{acm2}
c_{-2}=
 	\begin{cases}
 	0 & \text { if } k\leq 2,\\
    -\frac{k(k-1)}{3(m+1)(m+2)} & \text { if } k > 2,\\
 	\end{cases}
\end{equation}
\begin{equation}\label{acm1}
c_{-1}=
 	\begin{cases}
 	0 & \text { if } k\leq 1,\\
    \frac{k}{(m+1)} & \text { if } k > 1,\\
 	\end{cases}
\end{equation}
\begin{equation}\label{ac0}
c_{0}=-\frac{1}{3(m+1)(m+2)},
\end{equation}
\begin{equation}\label{ac1}
c_{1}=-\frac{\lambda}{3(m+1)(m+2)(k+1)}.
\end{equation}
Let $\sigma=\sigma(n,r)=\sum_{i=1}^n j_{r,i}$. From \eqref{aPi}-\eqref{ac1} one can easily see that for each $n$ and $r$
\begin{equation}\label{amonomials}
R_{r,n}st=
 	\begin{cases}
 	0 & \text { if } \sigma <-1,\\
    C_{r,n}s^\beta t^{\sigma+1} & \text { if } \sigma \geq -1\\
 	\end{cases}
\end{equation}
where $n+1\leq \beta\leq 2n+1$ and $C_{r,n}$ is a constant which only depends on $n$ and $r$.

Let $\tilde{\lambda}=\max\{1,\lambda\}$. We claim that for each $n$ and $r$,
\begin{equation}\label{aclaim}
|C_{r,n}|\leq \frac{\tilde{\lambda}^n}{(n+1)!(\sigma+1)!}.
\end{equation}
Taking $m=1$, $k=1$ in  \eqref{aPi}-\eqref{ac1}, one can check that the claim holds for $n=1$. Suppose it holds for $n=\ell-1$ and for all $r \in \{1,2,.. ,4^{\ell -1}\}$. Then for $n=\ell$ and $r^* \in \{1,2,.. ,4^{\ell}\}$, using \eqref{aPi} and \eqref{amonomials}, we obtain $R_{r^*,\ell}st=P_i R_{r,\ell-1}st= C_{r,\ell-1}P_i s^\beta t^{\sigma+1}=C_{r,\ell-1}c_i s^{\beta^*} t^{\sigma^*+1}$
for some $i\in\{-2,-1,0,1\}$ and $r \in \{1,2,.. ,4^{\ell -1}\}$, where $\beta^*$ is either $\beta+1$ or $\beta+2$, $\sigma^*=\sigma +i$. By the induction assumption $C_{r,\ell-1}\leq \frac{\tilde{\lambda}^{\ell-1}}{\ell!(\sigma+1)!}.$  Moreover \eqref{acm2}-\eqref{ac1} and the fact that $\beta\geq \ell$  imply $|c_i|\leq \frac{\sigma+1}{\ell+1}$ for $i=-1,-2$, $|c_0| <\frac{1}{\ell+1}$, and $|c_1|< \frac{\lambda}{(\sigma+2)(\ell+1)}$. Hence for each $i\in \{-2,-1,0,1\}$ we get $|C_{r^*,\ell}|= |C_{r,(\ell-1)}c_i| \leq \frac{\tilde{\lambda}^{\ell}}{(\ell+1)!(\sigma+i+1)!}=\frac{\tilde{\lambda}^{\ell}}{(\ell+1)!(\sigma^*+1)!}$, which proves that the claim holds for $n=\ell$ as well.

By \eqref{aproduct}, \eqref{amonomials}, \eqref{aclaim} and the fact that $0\leq s, t \leq L$ in the triangle $T_0$, we obtain \begin{equation}\label{Hnest0}\|H^n\|_{\infty}\leq \frac{4^n\tilde{\lambda}^nL^{3n+2}}{(n+1)!}\end{equation} which is summable. Moreover, since $H^n$ is a linear combination of $4^n$ monomials of the form $s^\beta t^{\sigma+1}$ with $\beta\leq 2n+1$ and $\sigma\leq n$,
any partial derivative $\partial^a_s \partial^b_t H^n$ of $H^n$ will be absolutely less than \begin{equation}\label{Hnest}\displaystyle\frac{(2n+1)^a (n+1)^b 4^n\tilde{\lambda}^nL^{3n+2-a-b}}{(n+1)!}\end{equation} which is also summable.
\end{proof}
Now, we define the pseudo-kernel by \begin{equation}\label{ktilde}\tilde{k}(x,y):=\tilde{G}(x-y,y)\end{equation} and consider the transformation given by \begin{equation}\label{mod-transform}\tilde{w}(x,t)\equiv u(x,t)-\int_0^x\tilde{k}(x,y)u(y,t)dy.\end{equation}
\begin{figure}[H]
  \centering
   \includegraphics[scale=0.75]{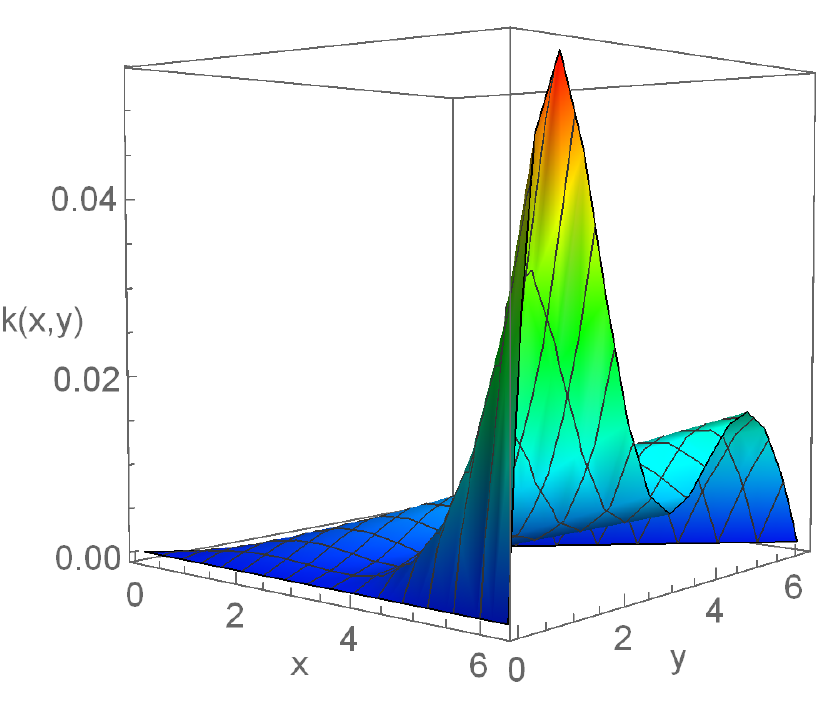}
  \caption{Pseudo-kernel $\tilde{k}$ when $\lambda=0.01$ ($L=2\pi$)}\label{Impk}
\end{figure}
\begin{figure}[H]
  \centering
   \includegraphics[scale=0.75]{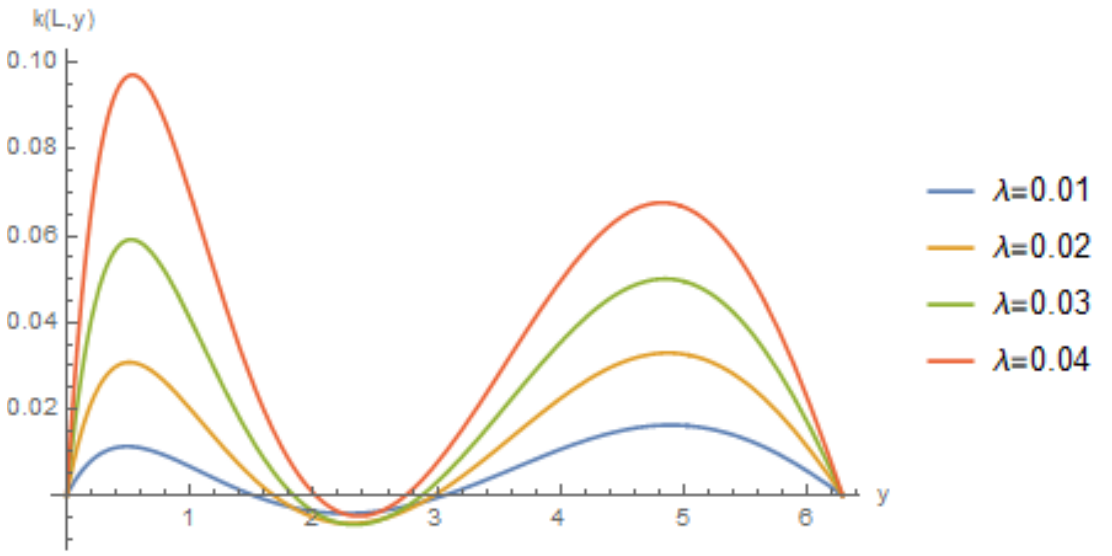}
  \caption{Control effort at the Dirichlet b.c. for different $\lambda$  ($L=2\pi$)}\label{effortk1y}
\end{figure}
\begin{figure}[H]
  \centering
   \includegraphics[scale=0.75]{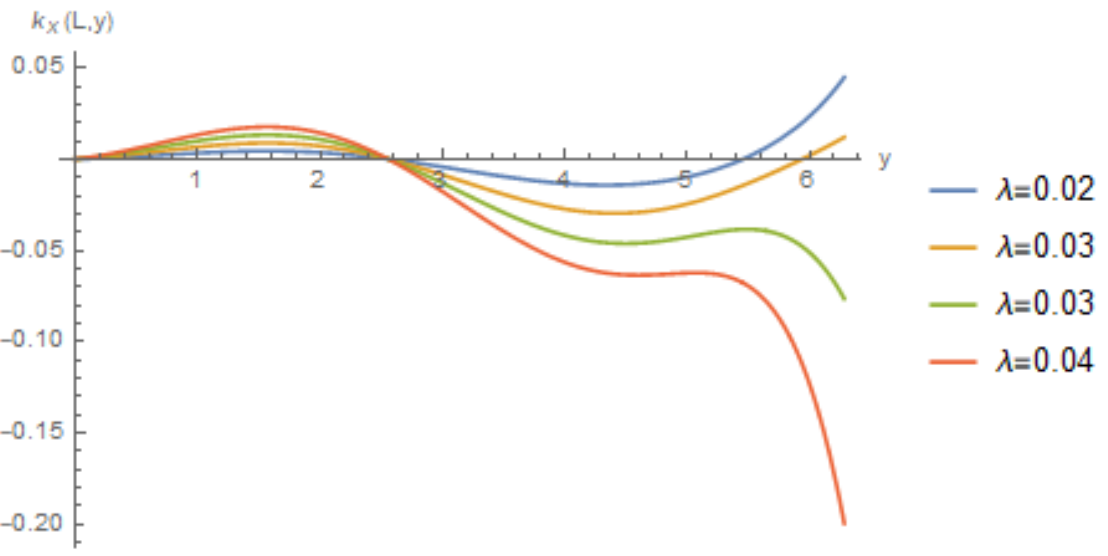}
  \caption{Control effort at the Neumann b.c. for different $\lambda$ ($L=2\pi$)}\label{effortky1y}
\end{figure}
Note that we have $\tilde{u}_x(0,t)=\tilde{w}_x(0,t)$ by the boundary conditions of $\tilde{k}$.  Using this fact, we can rewrite the modified target system as
\begin{equation}\label{HomKdVBurgers-1}
 	\begin{cases}
 	\displaystyle \tilde{w}_{t} +\tilde{ w}_{x} + \tilde{w}_{xxx} + \lambda \tilde{w} = \tilde{k}_y(x,0)\tilde{w}_x(0,t) & \text { in } \Omega\times \mathbb{R_+},\\
    \tilde{w}(0,t) = \tilde{w}(L,t) = \tilde{w}_{x}(L,t) = 0 & \text { in } \mathbb{R_+},\\
    \tilde{w}(x,0)=\tilde{w}_0(x):= u_0-\int_0^x\tilde{k}(x,y)u_0(y)dy & \text { in } \Omega.
 	\end{cases}
\end{equation}
Multiplying the above model by $\tilde{w}$ and integrating over $(0,1)$, using the Cauchy-Schwarz inequality, we obtain
\begin{multline}\label{HomKdVBurgers-2}
 	\frac{1}{2}\frac{d}{dt}\|\tilde{w}(t)\|_{L^2(\Omega)}^2+\lambda\|\tilde{w}(t)\|_{L^2(\Omega)}^2
 \le -\frac{1}{2}|\tilde{w}_x(0,t)|^2+\int_0^L\tilde{k}_y(x,0)\tilde{w}_x(0,t)\tilde{w}(x,t)dx\\
 \le \cancel{-\frac{1}{2}|\tilde{w}_x(0,t)|^2}+\cancel{\frac{1}{2}|\tilde{w}_x(0,t)|^2}+\frac{1}{2}\left(\int_0^L|\tilde{k}_y(x,0)||\tilde{w}(x,t)|dx\right)^2.
\end{multline}
Since $\tilde{k}$ is smooth on the compact set $\mathcal{T}$, we  have
\begin{equation}\label{HomKdVBurgers-3}
 	\frac{1}{2}\frac{d}{dt}\|\tilde{w}(t)\|_{L^2(\Omega)}^2+ \left(\lambda-\frac{1}{2}\|\tilde{k}_y(\cdot,0)\|_{L^2(\Omega)}^2\right)\|\tilde{w}(t)\|_{L^2(\Omega)}^2\le 0.
\end{equation}
It follows that
\begin{equation}\label{HomKdVBurgers-4}
 	\|\tilde{w}(t)\|_{L^2(\Omega)}^2\le \|\tilde{w}_0\|_{L^2(\Omega)}^2e^{-2{\alpha}t} ,
\end{equation} where
${\alpha}\equiv \lambda-\frac{1}{2}\|\tilde{k}_y(\cdot,0)\|_{L^2(\Omega)}^2.$
The graph of the function $\tilde{k}_y(\cdot,0)$ is depicted in Figure \ref{ky0tilde} on a domain of length $L=2\pi$.
\begin{figure}[H]
  \centering
   \includegraphics[scale=0.75]{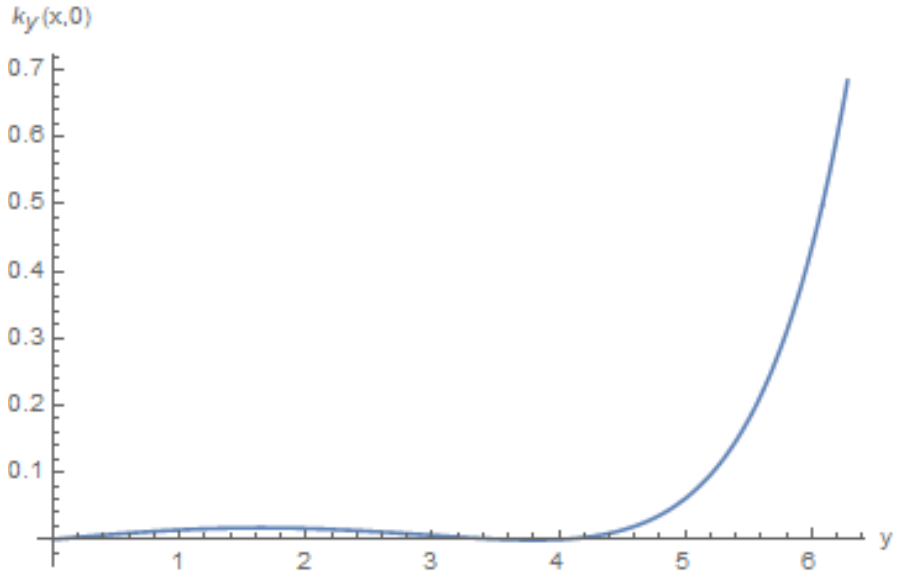}
  \caption{Pseudo-kernel $\tilde{k}$ when $\lambda=0.03$ ($L=2\pi$)}\label{ky0tilde}
\end{figure}

By taking $L^2(\Omega)$ norms of both sides of \eqref{mod-transform} (with $t=0$) and using the Cauchy-Schwarz inequality, we get \begin{equation}\label{w0u0} \|\tilde{w}_0\|_{L^2(\Omega)}\le \left(1+\|\tilde{k}\|_{L^2(T)}\right) \|u_0\|_{L^2(\Omega)}.\end{equation}

Let $K:H^l(\Omega)\rightarrow H^l(\Omega)$ ($l\ge 0$) be the integral operator defined by $(K\varphi)(x):=\int_0^x\tilde{k}(x,y)\varphi(y)dy.$  It is not difficult to prove that the operator $I-K$ is invertible from $H^l(\Omega)\rightarrow H^l(\Omega)$ (for $l\ge 0$) with a bounded inverse.  This is proved in a general setting in the lemma below:

\begin{lem}\label{inverselem}
  $I-K$ is invertible with a bounded inverse from $H^l(\Omega)\rightarrow H^l(\Omega)$ ($l\ge 0$).
\end{lem}
\begin{rem}
  The above lemma can be expressed in a sharper form.  Indeed, the proof below shows that $\Phi$ is a bounded operator from $L^2(\Omega)\rightarrow H^l(\Omega)$ ($l=0,1,2$) and it is a bounded operator from $H^{l-2}(\Omega)\rightarrow H^{l}(\Omega)$ ($l> 2$).
\end{rem}
\begin{proof}
The above lemma can be proven by slightly modifying the proof of \cite[Lemma 2.4]{Liu03}.  However, we will still give a brief proof here since we will need to refer to some crucial details of the proof of this lemma later in the proofs of the stabilization and well-posedness results.

To this end, let us first consider the case $l=0$ and let $\psi=(I-K)\varphi$ for some $\varphi\in L^2(\Omega)$.  The idea is to first write $\psi=\varphi-v$ where $v=K\varphi.$  Note that then, $$\psi(x)=\varphi(x)-[K\varphi](x)=(\psi(x)+v(x))-\int_0^x\tilde{k}(x,y)(\psi(y)+v(y))dy.$$  This gives $$v(x)=\int_0^x\tilde{k}(x,y)\psi(y)dy+\int_0^x\tilde{k}(x,y)v(y)dy.$$  Given a fixed $\psi$, one can solve this equation via succession (see \cite[Lemma 2.4]{Liu03} for the details of the succession argument).  This implicitly defines a linear operator $\Phi:\psi\mapsto v$ on $L^2(\Omega)$ with the property that $\Phi$ is bounded, i.e., there exists $C_0>0$ such that \begin{equation}\label{l0}\|v\|_{L^2(\Omega)}\le C_0\|\psi\|_{L^2(\Omega)},\end{equation} where $C_0$ depends only on $\|\tilde{k}\|_{L^\infty(\mathcal{T})}$.  But then, $\varphi$ is simply equal to $(I+\Phi)\psi$, and therefore $(I-K)^{-1}$ exists, equals $I+\Phi$, and is bounded.  By differentiating and using the smoothness of $\tilde{k}$, $(I-K)^{-1}$ extends to a linear bounded operator also on Sobolev spaces $H^l(\Omega)$ ($l\ge 1$). Indeed, since $\tilde{k}(x,x)=0$, we have
\begin{equation}\label{vxx}v_x(x)=\int_0^x\tilde{k}_x(x,y)(\psi(y)+v(y))dy,\end{equation} which implies
$\|v_x\|_{L^2(\Omega)}\le \|\tilde{k}_x\|_{L^2(\mathcal{T})}\left(\|\psi\|_{L^2(\Omega)}+\|v\|_{L^2(\Omega)}\right).$ Hence, using \eqref{l0}, we have \begin{equation}\label{l1}\|v\|_{H^1(\Omega)}\le C_1\|\psi\|_{L^2(\Omega)},\end{equation} where $C_1$ depends on $\|\tilde{k}_x\|_{L^2(\mathcal{T})}$ and  $C_0$. This shows that $\Phi$ is bounded from $L^2(\Omega)$ into $H^1(\Omega)$, a fortiori bounded from $H^1(\Omega)$ into $H^1(\Omega)$.  Now for $l=2$, using $k_x(x,x)=\frac{\lambda}{3}x$, $(\partial_x^2v)(x)=\frac{\lambda}{3}x(\psi(x)+v(x))+\int_0^x(\partial_x^2\tilde{k})(x,y)(\psi(y)+v(y))dy.$ Taking $L^2(\Omega)$ norms of both sides and using the previous inequalities, we get $\|v\|_{H^2(\Omega)}\le C_2\|\psi\|_{L^2(\Omega)},$ where $C_2$ depends on $\|\partial_x^2\tilde{k}\|_{L^2(\mathcal{T})}$, $C_1$, and $\lambda$.   This shows that $\Phi$ is bounded from $L^2(\Omega)$ into $H^2(\Omega)$, a fortiori bounded from $H^1(\Omega)$ or $H^2(\Omega)$ into $H^2\Omega)$.  Proceeding in the same fashion, one can show that $\|v\|_{H^3(\Omega)}\le C_3\|\psi\|_{H^1(\Omega)},$ where $C_3$ is a fixed constant depending on various norms of $\tilde{k}$.  More generally,
$\|v\|_{H^l(\Omega)}\le C_l\|\psi\|_{H^{l-2}(\Omega)},$ where $l> 2$ and $C_l$ depends on  various norms of $\tilde{k}$.  Hence, for $l> 2$, $\Phi$ is a bounded operator from $H^{l-2}(\Omega)$ into $H^l(\Omega)$, and a fortiori bounded from $H^{l}(\Omega)$ into $H^l(\Omega)$.
\end{proof}
\begin{rem} Another important estimate that follows from \eqref{vxx} via \eqref{l0} is that
\begin{equation}\label{linftyrem}\|v_x\|_{L^\infty(\Omega)}\le C\|\psi\|_{L^2(\Omega)}\end{equation} for some $C>0$ that depends on $\tilde{k}$.
\end{rem}

From the above lemma, it follows in particular that $u(x,t)=[(I-K)^{-1}\tilde{w}](x,t)$, and moreover \begin{equation}\label{ulessw}\|u(t)\|_{L^2(\Omega)}\le \|(I-K)^{-1}\|_{B[L^2(\Omega)]}\cdot \|\tilde{w}(t)\|_{L^2(\Omega)},\end{equation} where $\|\cdot\|_{B[L^2(\Omega)}$ is the operator norm of $(I-K)^{-1}$ from $L^2(\Omega)$ into $L^2(\Omega)$.

Combining \eqref{ulessw} with \eqref{HomKdVBurgers-4} and \eqref{w0u0}, we conclude that

\begin{equation}\label{linshot}\|u(t)\|_{L^2(\Omega)}\le \left(1+\|\tilde{k}\|_{L^2(T)}\right)\|(I-K)^{-1}\|_{B[L^2(\Omega)]}\,\|u_0\|_{L^2(\Omega)}e^{-\alpha t}.\end{equation}

We can prove that the parameter $\alpha$ in the above estimate is positive if $\lambda$ is sufficiently small. Indeed, we have the following lemma.
\begin{lem}\label{alemlambda}For a given $L$, there exists sufficiently small $\lambda$ such that $\alpha=\lambda-\frac{1}{2}\|k_y(\cdot, 0)\|^2_{L^2}>0.$
\end{lem}
\begin{proof}
Taking the partial derivative of both sides of \eqref{aCauchy} with respect to $t$ and taking $i=0$ we see that
$\tilde{G}^j_t(s,t)=\frac{\lambda}{3}\sum_{n=0}^{j-1}H^n_t(s,t).$ Passing to the limit we obtain $\tilde{G}_t(s,t)=\frac{\lambda}{3}\sum_{n=0}^{\infty}H^n_t(s,t).$ Note that for $\lambda<1$, $\tilde{\lambda}=1$. Therefore by \eqref{Hnest0} the summation term is absolutely less than some constant $M$ that only depends on $L$. Hence we get $\|\tilde{G}_t\|_{\infty}\leq\frac{\lambda M}{3}.$ Since $k_y(x,0)=\tilde{G}_t(s,0)$, in particular we have $\|k_y(\cdot, 0)\|^2_{L^2}\leq L \|k_y(\cdot, 0)\|^2_{\infty}\leq L \|\tilde{G}_t\|^2_{\infty}\leq \frac{\lambda^2 M^2 L}{9}.$ As a result, $\alpha=\lambda-\frac{1}{2}\|k_y(\cdot, 0)\|^2_{L^2}\geq \lambda-\frac{\lambda^2 M^2 L}{18}=\lambda^2(\frac{1}{\lambda}-\frac{ M^2 L}{18})$ which is positive for sufficiently small $\lambda$.
\end{proof}

The inequality \eqref{linshot} together with Lemma \ref{alemlambda} proves the linear part of Theorem \ref{Linthm}.

\subsection{Nonlinear model}\label{ProofofThm2}
In this section, we consider the nonlinear KdV model \eqref{nonlinearKdV} with the feedback controllers given in \eqref{controllers}.  By using the transformation given in \eqref{mod-transform}, we obtain the following PDE from \eqref{KdVBurgers}, noting that $\tilde{k}(x,x)=0$:
\begin{equation}\label{ch414}
\tilde{w}_{t} + \tilde{w}_{x} + \tilde{w}_{xxx} + \lambda \tilde{w}
= \tilde{k}_y(\cdot,0)\tilde{w}_x(0,\cdot)-(I-K)[\left(\tilde{w}+ v\right)\left(\tilde{w}_{x} + v_x\right)]
\end{equation}
with homogeneous boundary conditions
\begin{equation}
  \tilde{w}(0,t) = 0 \; , \; \tilde{w}(L,t) = 0, \quad \textrm{and} \quad \tilde{w}_{x}(L,t) = 0,
\end{equation} where $v(x,t)=[\Phi\tilde{w}](x,t)$, with $\Phi$ being the linear operator defined in Section \ref{ProofofThm} in the proof of Lemma \ref{inverselem}.
Multiplying \eqref{ch414} by $\tilde{w}(x,t)$ and integrating over $\Omega=(0,L)$, we obtain
\begin{multline}\label{ch416}
\int_{0}^{L}\tilde{w}(x,t)\tilde{w}_{t}(x,t)dx = \int_0^L\tilde{k}_y(x,0)\tilde{w}_x(0,t)\tilde{w}(x,t)dx-\int_{0}^{L}\tilde{w}(x,t)\tilde{w}_{x}(x,t)dx \\
- \int_{0}^{L}\tilde{w}(x,t)\tilde{w}_{xxx}(x,t)dx
  -\lambda\int_{0}^{L}\tilde{w}^{2}(x,t)dx - \int_{0}^{L}\tilde{w}^2(x,t)\tilde{w}_x(x,t)dx - \int_{0}^{L}\tilde{w}^2(x,t)v_x(x,t)dx \\
  -\int_{0}^{L}\tilde{w}(x,t)\tilde{w}_x(x,t)v(x,t)dx-\int_{0}^{L}\tilde{w}(x,t)v(x,t)v_x(x,t)dx\\
  +\int_0^L\left(\int_0^x\tilde{k}(x,y)\tilde{w}(y,t)\tilde{w}_y(y,t)dy\right)\tilde{w}(x,t)dx
  +\int_0^L\left(\int_0^x\tilde{k}(x,y)\tilde{w}(y,t)\tilde{v}_y(y,t)dy\right)\tilde{w}(x,t)dx\\
  +\int_0^L\left(\int_0^x\tilde{k}(x,y)\tilde{v}(y,t)\tilde{w}_y(y,t)dy\right)\tilde{w}(x,t)dx
  +\int_0^L\left(\int_0^x\tilde{k}(x,y)\tilde{v}(y,t)\tilde{v}_y(y,t)dy\right)\tilde{w}(x,t)dx.
\end{multline}

We estimate the last four terms at the right hand side of \eqref{ch416} as follows:
\begin{multline}\label{four-1}
\int_0^L\left(\int_0^x\tilde{k}(x,y)\tilde{w}(y,t)\tilde{w}_y(y,t)dy\right)\tilde{w}(x,t)dx = \frac{1}{2}\int_0^L\left(\int_0^x\tilde{k}(x,y)\frac{\partial}{\partial y}\tilde{w}^2(y,t)dy\right)\tilde{w}(x,t)dx\\
=\frac{1}{2}\int_0^L\left.\tilde{k}(x,y)\tilde{w}^2(y,t)\right|_{0}^x\tilde{w}(x,t)dx- \frac{1}{2}\int_0^L\left(\int_0^x\tilde{k}_y(x,y)\tilde{w}^2(y,t)dy\right)\tilde{w}(x,t)dx\\
\le \frac{\sqrt{L}}{2}\|\tilde{k}_y\|_{L^\infty(T_0)}\|\tilde{w}(t)\|_{L^2(\Omega)}^3,
\end{multline}

\begin{equation}
\int_0^L\left(\int_0^x\tilde{k}(x,y)\tilde{w}(y,t){v}_y(y,t)dy\right)\tilde{w}(x,t)dx \le \|\tilde{k}\|_{L^2(T_0)}\|v_x(t)\|_{L^\infty(\Omega)}\|\tilde{w}(t)\|_{L^2(\Omega)}^2,
\end{equation}

\begin{multline}
\int_0^L\left(\int_0^x\tilde{k}(x,y){v}(y,t)\tilde{w}_y(y,t)dy\right)\tilde{w}(x,t)dx \\
=\int_0^L\left.\tilde{k}(x,y){v}(y,t)\tilde{w}(y,t)\right|_{0}^x\tilde{w}(x,t)dx- \int_0^L\left(\int_0^x\tilde{k}_y(x,y){v}(y,t)\tilde{w}(y,t)dy\right)\tilde{w}(x,t)dx\\
- \int_0^L\left(\int_0^x\tilde{k}(x,y){v}_y(y,t)\tilde{w}(y,t)dy\right)\tilde{w}(x,t)dx\le \sqrt{L}\|\tilde{k}_y\|_{L^\infty(T_0)}\|v(t)\|_{L^2(\Omega)}\|\tilde{w}(t)\|_{L^2(\Omega)}^2\\
+\|\tilde{k}\|_{L^2(T_0)}\|v_x(t)\|_{L^\infty(\Omega)}\|\tilde{w}(t)\|_{L^2(\Omega)}^2,
\end{multline}

\begin{multline}\label{four-4}
\int_0^L\left(\int_0^x\tilde{k}(x,y)v(y,t)v_y(y,t)dy\right)\tilde{w}(x,t)dx = \frac{1}{2}\int_0^L\left(\int_0^x\tilde{k}(x,y)\frac{\partial}{\partial y}v^2(y,t)dy\right)\tilde{w}(x,t)dx\\
=\frac{1}{2}\int_0^L\left.\tilde{k}(x,y)v^2(y,t)\right|_{0}^x\tilde{w}(x,t)dx- \frac{1}{2}\int_0^L\left(\int_0^x\tilde{k}_y(x,y)v^2(y,t)dy\right)\tilde{w}(x,t)dx\\
\le \frac{\sqrt{L}}{2}\|\tilde{k}_y\|_{L^\infty(T_0)}\|v(t)\|_{L^2(\Omega)}^2\|\tilde{w}(t)\|_{L^2(\Omega)}.
\end{multline}

Now, estimating the other terms using integration by parts and the Cauchy-Schwarz inequality, and combining these with \eqref{four-1}-\eqref{four-4}, it follows that
\begin{multline}\label{ch417}
\frac{1}{2}\frac{d}{dt}\|\tilde{w}(t)\|_{L^2(\Omega)}^2 + \left(\lambda-\frac{1}{2}\|\tilde{k}_y(\cdot,0)\|_{L^2(\Omega)}^2\right)\|\tilde{w}(t)\|_{L^2(\Omega)}^2 \\
\le \left(\frac{3}{2}+2\|\tilde{k}\|_{L^2(T_0)}\right)\|\tilde{w}(t)\|_{L^2(\Omega)}^2\|v_x(t)\|_{L^\infty(\Omega)}+\|\tilde{w}(t)\|_{L^2(\Omega)}\|v(t)\|_{L^2(\Omega)}\|v_x(t)\|_{L^\infty(\Omega)}\\
+\frac{\sqrt{L}}{2}\|\tilde{k}_y\|_{L^\infty(T_0)}\|\tilde{w}(t)\|_{L^2(\Omega)}^3
+\sqrt{L}\|\tilde{k}_y\|_{L^\infty(T_0)}\|v(t)\|_{L^2(\Omega)}\|\tilde{w}(t)\|_{L^2(\Omega)}^2\\
+\frac{\sqrt{L}}{2}\|\tilde{k}_y\|_{L^\infty(T_0)}\|v(t)\|_{L^2(\Omega)}^2\|\tilde{w}(t)\|_{L^2(\Omega)}.
\end{multline}

Using \eqref{l0} and \eqref{linftyrem}, we deduce the following inequality:
\begin{equation}\label{bernoulli}
  y'+2\alpha y-cy^\frac{3}{2}\le 0,
\end{equation} where $y(t)\equiv \|\tilde{w}(t)\|_{L^2(\Omega)}^2$, and $c$ is a constant that depends on $L$ and various norms of $\tilde{k}$.  Solving the inequality \eqref{bernoulli} and assuming $\displaystyle \|\tilde{w}_0\|_{L^2(\Omega)}<\frac{\alpha}{c}$, we get
\begin{equation}\label{Nonlinw}
  \|\tilde{w}(t)\|_{L^2(\Omega)}^2=y(t)\le \frac{1}{\left[\left(\frac{1}{\|\tilde{w}_0\|_{L^2(\Omega)}}-\frac{c}{2\alpha}\right)e^{\alpha t}+\frac{c}{2\alpha}\right]^2}<\frac{1}{\left[\frac{e^{\alpha t}}{2\|\tilde{w}_0\|_{L^2(\Omega)}}\right]^2}.
\end{equation}  Recall that $\|\tilde{w}_0\|_{L^2(\Omega)}\lesssim\|u_0\|_{L^2(\Omega)}$.   Combining this with \eqref{ulessw} and \eqref{Nonlinw}, we deduce
\begin{equation}
  \|u(t)\|_{L^2(\Omega)}\lesssim \|u_0\|_{L^2(\Omega)}e^{-\alpha t}, \text{ for } t\ge 0.
\end{equation}  Hence, the proof of Theorem \ref{Linthm} for the nonlinear KdV equation is also complete.  Note that the smallness assumption on the initial datum $\tilde{w}_0$ implies a smallness assumption on $u_0$ due to the fact that we also have $\|u_0\|_{L^2(\Omega)}\lesssim\|\tilde{w}_0\|_{L^2(\Omega)}$ thanks to Lemma \ref{inverselem}.
\section{Well-posedness}In this section, we prove the well-posedness of the PDE models studied in the previous sections.  For simplicity, we assume $L=1$ throughout this section. This assumption has no consequence as far as wellposedness is concerned, and all results proved here are also true for any $L>0$.  Thanks to Lemma \ref{lemback}, it is enough to prove the well-posedness of the respective modified target systems in order to obtain well-posedness of \eqref{KdVBurgers} and \eqref{nonlinearKdV}.
\subsection{Linearised model}
Consider the following linear KdV equation with homogeneous boundary conditions.
\begin{equation}\label{KdV-wp}
 	\begin{cases}
 	\displaystyle y_{t} +y_{x} + y_{xxx} + {\lambda} y = a(x)y_x(0,\cdot) & \text { in } \Omega\times \mathbb{R_+},\\
    {y}(0,t) = {y}(1,t) = {y}_{x}(1,t) = 0 & \text { in } \mathbb{R_+},\\
    {y}(x,0)=y_0\in L^2(\Omega) & \text { in } \Omega.
 	\end{cases}
\end{equation} We have the following result.

\begin{prop}\label{wellposednessprop1}
\begin{enumerate}
  \item[i)] Let $T'>0$ be arbitrary and $y_0,a\in L^2(\Omega)$.  Then, there exists $T\in (0,T')$ independent of the size of $y_0$ such that \eqref{KdV-wp} has a unique local solution $y\in C([0,T];L^2(\Omega))\cap L^2(0,T;H^1(\Omega))$ satisfying also $y_x\in C([0,1];L^2(0,T)).$ Moreover, if $a\in L^\infty(\Omega)$, then $y$ extends as a global solution. In other words, $T$ can be taken as $T'$.
  \item[ii)] Let $a\in H^1(\Omega)$ and let $y_0\in H^3(\Omega)$ satisfy the compatibility conditions $y_0(0)=y_0(1)=y_0'(1)=0.$  Then, the (local/global) solution in part (i) enjoys the extra regularity $y\in C([0,T];H^3(\Omega))\cap L^2(0,T;H^4(\Omega)).$
\end{enumerate}
\end{prop}
\begin{proof}

\textbf{Step 1 - Local wellposedness:} Let us define the linear operator $A:D(A)\subset L^2(\Omega)\rightarrow L^2(\Omega)$ by $A\varphi := -\varphi'-\varphi''',$ where $D(A):=\{\varphi\in H^3(\Omega):\varphi(0)=\varphi(1)=\varphi'(1)=0\}.$  Then, the initial boundary value problem \eqref{KdV-wp} can be rewritten in the abstract operator theoretic form
\begin{equation}\label{KdV-wp-abs}
 	\begin{cases}
 	\displaystyle \dot{y} = Ay+Fy,\\
    {y}(0)=y_0,
 	\end{cases}
\end{equation} where $F\varphi:=-\lambda \varphi + a(\cdot)\gamma_1^0\varphi$.  Here, $\gamma_1^0$ is the first order trace operator at the left endpoint, i.e., $\gamma_1^0\varphi:=\varphi'(0)$.  This operator is well-defined for $\varphi\in H^{\frac{3}{2}+\epsilon}(\Omega)\supset D(A)$.

It is not difficult to see that the adjoint of $A$ is defined by $A^*\varphi:= \varphi'+\varphi'''$ with $D(A^*):=\{\varphi\in H^3(\Omega):\varphi(0)=\varphi(1)=\varphi'(0)=0\}.$

$A$ is a densely defined closed operator, and moreover, $A$ and $A^*$ are dissipative \cite[Proposition 3.1]{Rosier1997}.  Therefore, $A$ generates a strongly continuous semigroup of contractions  $\displaystyle\{S(t)\}_{t\ge 0}$ on $L^2(\Omega)$ \cite[Corollary I.4.4]{Pazy}. Now we construct the operator \begin{equation}\label{soloperator}y=[\Psi z](t):= S(t)y_0+\int_0^tS(t-s)Fz(s)ds.\end{equation}

Let us define the space (see e.g., \cite{BSZ03}) \begin{equation}\label{ourspace}Y_T:=\{z\in C([0,T];L^2(\Omega))\cap L^2(0,T;H^1(\Omega))\,|\,z_x\in C([0,1];L^2(0,T))\}\end{equation} equipped with the norm
$\|z\|_{Y_T}:=\left(\|z\|_{C([0,T];L^2(\Omega))}^2+\|z\|_{L^2(0,T;H^1(\Omega))}^2+\|z_x\|_{C([0,1];L^2(0,T))}^2\right)^{\frac{1}{2}}.$  Then, for $z\in Y_T$\,, by using the semigroup estimates \cite[Prop 2.1, Prop 2.4, Prop 2.16, Prop 2.17]{BSZ03}, we have
\begin{multline}\label{estimate01}\|y\|_{Y_T}=\|\Psi z\|_{Y_T}\le \|S(t)y_0\|_{Y_T}+\left\|\int_0^tS(t-s)Fz(s)ds\right\|_{Y_T}\\
\le c_0(1+T)^\frac{1}{2}\|y_0\|_{L^2(\Omega)}+c_1(1+T)^\frac{1}{2}\left\|-\lambda z+az_x(0,\cdot)\right\|_{L^1(0,T;L^2(\Omega))}\\
\le c_0(1+T)^\frac{1}{2}\|y_0\|_{L^2(\Omega)}+c_1(1+T)^\frac{1}{2}\sqrt{T}(1+\|a\|_{L^2(\Omega)})\|z\|_{Y_T}\,,\end{multline} where $c_0$ and $c_1$ are positive constants which do not depend on the varying parameters.  It follows that $\Psi$ maps $Y_T$ into itself.  Now, let $z_1,z_2\in Y_T$ and $y_1=\Psi z_1$, $y_2=\Psi z_2$.  By using similar arguments, we have
\begin{equation*}
  \|y_1-y_2\|_{Y_T}=\|\Psi z_1-\Psi z_2\|_{Y_T}\le c_1(1+T)^\frac{1}{2}\sqrt{T}(1+\|a\|_{L^2(\Omega)})\|z_1-z_2\|_{Y_T}.
\end{equation*}  Let $T\in (0,T')$ be such that $0< (1+T)^\frac{1}{2}\sqrt{T}< \left(\frac{1}{c_1\left(1+\|a\|_{L^2(\Omega)}\right)}\right).$ Then, $\Psi$ is a contraction on $Y_T$, and this gives us a unique local solution $y\in Y_T$.  Here, the size of $T$ is independent of the size of the initial datum. This contrasts with the corresponding nonlinear model in which the size of $T$ is related to the size of the initial datum.

\textbf{Step 2 - Global wellposedness: } Let $T_{\max}\le T'$ be the maximal time of existence for the local solution found in Step 1 in the sense that $y\in Y_T$ for all $T<T_{\max}$.   In order to prove that $y$ is global, and deduce that $T$ can be taken as $T'$, it is enough to show that $\displaystyle \lim_{T\rightarrow T_{\max}^-}\|y\|_{Y_T}<\infty.$  This will be proved via multipliers, which will be done only formally, but the calculations can always be justified by a density argument which relies on the regularity result in part (ii) of this proposition. To this end, we multiply  \eqref{KdV-wp} by $y$ and integrate over $\Omega$ to obtain
\begin{equation}\label{Iden01}
  \frac{1}{2}\frac{d}{dt}\|y(t)\|_{L^2(\Omega)}^2+\frac{1}{2}|y_x(0,t)|^2 + \lambda \|y(t)\|_{L^2(\Omega)}^2 = \int_0^1 a(x)y_x(0,t)y(x,t)dx.
\end{equation} Using $\epsilon$-Young's inequality with $\displaystyle\epsilon=\frac{1}{4}$, we have
\begin{equation}\label{Iden02}
  \frac{1}{2}\frac{d}{dt}\|y(t)\|_{L^2(\Omega)}^2+\frac{1}{4}|y_x(0,t)|^2 + \lambda \|y(t)\|_{L^2(\Omega)}^2 \le  \|a(x)\|_{L^\infty(\Omega)}^2\|y(t)\|_{L^2(\Omega)}^2.
\end{equation}  Integrating the above inequality over $(0,t)$, we get
\begin{equation}\label{Iden03}
 \|y(t)\|_{L^2(\Omega)}^2+\int_0^t|y_x(0,t)|^2dt + \le  2\|y_0\|_{L^2(\Omega)}^2+4(\|a(x)\|_{L^\infty(\Omega)}^2-\lambda)\int_0^t\|y(s)\|_{L^2(\Omega)}^2ds.
\end{equation} Let $E_0(t):=\|y(t)\|_{L^2(\Omega)}^2+\int_0^t|y_x(0,t)|^2dt.$ Then, from \eqref{Iden03}, we get
$$E_0(t)\le 2\|y_0\|_{L^2(\Omega)}^2+4\left|\|a(x)\|_{L^\infty}^2-\lambda\right|\int_0^tE_0(s)ds.$$ Now, thanks to the Gronwall's lemma, we have
\begin{equation}\label{Iden04}
 E_0(t)=\|y(t)\|_{L^2(\Omega)}^2+\int_0^t|y_x(0,t)|^2dt \le  2\|y_0\|_{L^2(\Omega)}^2e^{4\left|\|a(x)\|_{L^\infty}^2-\lambda\right|t},\, t\in [0,T_{\max}).
\end{equation} We in particular deduce that
\begin{equation}\label{Firstimpest}
                                            \lim_{T\rightarrow T_{\max}^-}\|y\|_{C([0,T];L^2(\Omega))}\le\sqrt{2}\|y_0\|_{L^2(\Omega)}e^{2\left|\|a(x)\|_{L^\infty}^2-\lambda\right|T_{max}}<\infty.
                                            \end{equation}
By using \eqref{Iden04}, we also deduce that
\begin{equation}\label{Iden05}
   \lim_{T\rightarrow T_{\max}^-}\|y\|_{L^2(0,T;L^2(\Omega))}\le \sqrt{2T_{max}}\|y_0\|_{L^2(\Omega)}e^{2\left|\|a(x)\|_{L^\infty}^2-\lambda\right|T_{max}}.
\end{equation}
Secondly, we multiply \eqref{KdV-wp} by $xy$ and integrate over $\Omega\times (0,t)$ and get
\begin{multline}\label{Iden06}
  \int_{0}^1xy^2(x,s)dx+3\int_0^t\int_0^1y_x^2(x,s)dxds+\lambda\int_0^t\int_0^1xy^2(x,s)dxds\\
   = \int_0^1xy_0^2(x)dx+\int_0^t\int_0^1y^2(x,s)dxds+ \int_0^t\int_0^1xay_x(0,s)y(x,s)dxds.
\end{multline} From the above identity, it follows that

\begin{equation}\label{Iden07}
   \|y_x\|_{L^2(0,t;L^2(\Omega))}^2\le \frac{1}{3}\|y_0\|_{L^2(\Omega)}^2+\left(\frac{1}{2}+\frac{\|a\|_{L^\infty(\Omega)}^2}{18}\right)\int_0^tE_0(s)ds.
\end{equation} Combining the above inequality with \eqref{Firstimpest}, we deduce that

\begin{multline}\label{Iden08}
   \lim_{T\rightarrow T_{\max}^-}\|y_x\|_{L^2(0,T;L^2(\Omega))}\\
   \le \frac{1}{\sqrt{3}}\|y_0\|_{L^2(\Omega)}+\left(\frac{1}{\sqrt{2}}+\frac{\|a\|_{L^\infty(\Omega)}}{3\sqrt{2}}\right)\sqrt{2T_{\max}}\left(\|y_0\|_{L^2(\Omega)}e^{2\left|\|a(x)\|_{L^\infty}^2-\lambda\right|T_{max}}\right).
\end{multline} Using \eqref{Iden05} and \eqref{Iden08}, we deduce that
\begin{multline}\label{Iden09}
   \lim_{T\rightarrow T_{\max}^-}\|y\|_{L^2(0,T;H^1(\Omega))}\\
   \le \frac{1}{\sqrt{3}}\|y_0\|_{L^2(\Omega)}+\left(1+\frac{1}{\sqrt{2}}+\frac{\|a\|_{L^\infty(\Omega)}}{3\sqrt{2}}\right)\sqrt{2T_{\max}}\left(\|y_0\|_{L^2(\Omega)}e^{2\left|\|a(x)\|_{L^\infty}^2-\lambda\right|T_{max}}\right)<\infty.
\end{multline}

Since $y$ is the fixed point in \eqref{soloperator}, we have \begin{equation}y=S(t)y_0+\int_0^tS(t-s)Fy(s)ds.\end{equation}
Using \cite[Prop 2.16 and Prop 2.17]{BSZ03}, we know that the semigroup enjoys the properties
\begin{equation}\label{semi-prop1}
  \sup_{x\in \Omega}\left\|\partial_x [S(t)y_0](x)\right\|_{L^2(0,T)}\le c_2\|y_0\|_{L^2(\Omega)}
\end{equation} and
\begin{equation}\label{semi-prop2}
  \sup_{x\in \Omega}\left\|\partial_x \left[\int_0^tS(t-s)Fy(s)ds\right](x)\right\|_{L^2(0,T)}\le c_3\int_0^T\left\|[Fy](\cdot,t)\right\|_{L^2(\Omega)}dt
\end{equation} for some $c_2,c_3>0.$ From the definition of $Fy$ we have
$$\|[Fy](\cdot,t)\|_{L^2(\Omega)}\le \lambda\|y(\cdot,t)\|_{L^2(\Omega)}+\|a\|_{L^2(\Omega)}|y_x(0,t)|.$$ Therefore, by \eqref{Iden04} and the Cauchy-Schwarz inequality, we have the estimate
\begin{multline}\label{semi-prop3}\int_0^T\left\|[Fy](\cdot,t)\right\|_{L^2(\Omega)}dt\le \lambda\int_0^T\sqrt{E_0(t)}dt+\|a\|_{L^2(\Omega)}\sqrt{T}\sqrt{E_0(T)}\\
\le \left(\lambda T_{\max}+\|a\|_{L^2(\Omega)}\sqrt{T_{\max}}\right)\sqrt{2}\|y_0\|_{L^2(\Omega)}e^{2\left|\|a(x)\|_{L^\infty}^2-\lambda\right|T_{max}}.\end{multline}
Now, it follows from \eqref{semi-prop1}-\eqref{semi-prop3} that
\begin{multline*}
   \lim_{T\rightarrow T_{\max}^-}\|y_x\|_{C([0,1];L^2(0,T))}\\
   \le c_2\|y_0\|_{L^2(\Omega)}+c_3\left(\lambda T_{\max}+\|a\|_{L^2(\Omega)}\sqrt{T_{\max}}\right)\sqrt{2}\|y_0\|_{L^2(\Omega)}e^{2\left|\|a(x)\|_{L^\infty}^2-\lambda\right|T_{max}}<\infty.
\end{multline*}

\textbf{Step 3 - Regularity:}
Regarding the regular solutions, assume that $y_0\in D(A)$ and consider the following problem:
\begin{equation}\label{KdV-wp-regular}
 	\begin{cases}
 	\displaystyle q_{t} +q_{x} + q_{xxx} + {\lambda} q = a(\cdot)q_x(0,\cdot) & \text { in } \Omega\times (0,T),\\
    {q}(0,t) = {q}(1,t) = {q}_{x}(1,t) = 0 & \text { in } (0,T),\\
    {q}(x,0)=q_0\equiv -y_0'(x)-y_0'''(x)-\lambda y_0(x)+y_0'(0)a(x) & \text { in } \Omega.
 	\end{cases}
\end{equation} Note that $q_0\in L^2(\Omega)$, and we can solve \eqref{KdV-wp-regular} in $Y_T$ as before.  Now, we set $y(x,t):=y_0(x)+\int_0^tq(x,s)ds.$ Then,
\begin{multline}\label{regular01}
  y_t(x,t)+y_x(x,t)+y_{xxx}(x,t)+\lambda y(x,t)-a(x)y_x(0,t)
  =q(x,t)+y_0'(x)+y_0'''(x)+\lambda y_0-y_0'(0)a(x)\\
  +\int_0^t\left(q_x(x,s)+q_{xxx}(x,s)+\lambda q(x,t)-a(x)q_x(0,s)\right)ds=0,
\end{multline} and moreover $y(x,0)=y_0$ and $y(0,t)=y(1,t)=y_x(1,t)=0.$  Therefore, $y$ solves \eqref{KdV-wp}. Writing
$$y_{xxx}(x,t)=-q(x,t)-y_x(x,t)-\lambda y(x,t)+a(x)y_x(0,t)$$ and taking $L^2(\Omega)$ norms of both sides we get
$$\|\partial_x^3 y(t)\|_{L^2(\Omega)}\le \|q(t)\|_{L^2(\Omega)}+\|\partial_x y(t)\|_{L^2(\Omega)}+\lambda\|y(t)\|_{L^2(\Omega)}+|y_x(0,t)|\|a\|_{L^2(\Omega)}.$$ Recall that we have the Gargliardo-Nirenberg inequalities $$\|\partial_x y(t)\|_{L^2(\Omega)}\lesssim \|y\|_{L^2(\Omega)}^\frac{2}{3}\|\partial_x^3y\|_{L^2(\Omega)}^\frac{1}{3}\text{ and }\|\partial_x^2 y(t)\|_{L^2(\Omega)}\lesssim \|y\|_{L^2(\Omega)}^\frac{1}{3}\|\partial_x^3y\|_{L^2(\Omega)}^\frac{2}{3},$$ and the trace inequality (remember that $y_x(1,t)=0$): $|y_x(0,t)|\le \|\partial_x^2 y\|_{L^2(\Omega)}.$

Using these estimates, we get $\|\partial_x^3 y(t)\|_{L^2(\Omega)}\lesssim \|q(t)\|_{L^2(\Omega)}+\|y(t)\|_{L^2(\Omega)}.$ By taking the sup norm with respect to the temporal variable, we deduce that $y\in C([0,T];H^3(\Omega)).$

Similarly, writing out $\partial_x^4y(x,t)=-q_x(x,t)-y_{xx}(x,t)-\lambda y_x(x,t)+a'(x)y_x(0,t),$ using the Gagliardo-Nirenberg and trace inequality, we get $\|\partial_x^4 y(t)\|_{L^2(\Omega)}\lesssim \|q_x(t)\|_{L^2(\Omega)}+\|y_x(t)\|_{L^2(\Omega)}.$ Taking $L^2(0,T)$ norms of both sides we deduce that $y\in L^2(0,T;H^4(\Omega)).$
\end{proof}
Global well-posedness of the linearized model \eqref{HomKdVBurgers-1} now follows from the Proposition \ref{wellposednessprop1} that we have just proved.

\begin{rem}
  One can interpolate between part (i) and part (ii) of the above proposition with respect to the smoothness of initial data and get the corresponding well-posedness and regularity result in fractional spaces.  For example, let $y_0\in H^s(\Omega)$ ($s\in [0,3]$) so that it satisfies the compatibility conditions $y_0(0)=y_0(1)=0$ if $s\in [0,3/2]$ and the compatibility conditions $y_0(0)=y_0(1)=y_0'(1)=0$ if $s\in (3/2,3].$ Then, with $a=a(x)$  sufficiently smooth, one has  $$y\in Y_T^s:= \{\psi\in C([0,T];H^s(\Omega))\cap L^2(0,T;H^{s+1}(\Omega))\,|\,\psi_x\in C([0,1];L^2(0,T))\}.$$  The arguments in Step 3 of the proof of the above proposition can be easily extended to the nonhomogeneous equation $\displaystyle y_{t} +y_{x} + y_{xxx} + {\lambda} y = a(x)y_x(0,\cdot)+f.$ One can first study this equation with $s=0$, $f\in L^1(0,T;L^2(\Omega)),$ and secondly with $s=3$, $f\in W^{1,1}(0,T;L^2(\Omega)).$ Then, by interpolation, for $s\in (0,3)$, one can get $y\in Y_T^s$ if $f\in W^{s/3,1}(0,T;L^2(\Omega))$. Moreover, the following estimates are true:
  \begin{equation}\label{YsT}
    \|y\|_{Y_{T}^s}\lesssim \|y_0\|_{H^s(\Omega)}+\|f\|_{W^{s/3,1}(0,T;L^2(\Omega))},
  \end{equation} and for $s=3$,
  \begin{equation}\label{YsT3}
    \|y_t\|_{Y_{T}}\lesssim \|y_0\|_{H^3(\Omega)}+\|f\|_{W^{1,1}(0,T;L^2(\Omega))}.
  \end{equation}
\end{rem}
\subsection{Nonlinear model}

Consider the following nonlinear KdV equation with homogeneous boundary conditions.
\begin{equation}\label{nonlinKdV-wp}
 	\begin{cases}
 	\displaystyle y_{t} +y_{x} + y_{xxx} + {\lambda} y = a(x)y_x(0,\cdot) -(I-K)[\left(y+v\right)\left(y_{x} + v_x\right)] & \text { in } \Omega\times \mathbb{R_+},\\
    {y}(0,t) = {y}(1,t) = {y}_{x}(1,t) = 0 & \text { in } \mathbb{R_+},\\
    {y}(x,0)=y_0\in L^2(\Omega) & \text { in } \Omega,
 	\end{cases}
\end{equation} where $v=\Phi(y)$, $\Phi$ being the linear operator defined in Section \ref{ProofofThm} in the proof of Lemma \ref{inverselem}.
\begin{prop}\label{wellposednessprop2}
\begin{enumerate}
  \item[i)] Let $T'>0$ be arbitrary and $y_0,a\in L^2(\Omega)$.  Then, there exists $T\in (0,T')$ depending on the size of $y_0$ such that \eqref{nonlinKdV-wp} has a unique local solution $y\in C([0,T];L^2(\Omega))\cap L^2(0,T;H^1(\Omega))$ satisfying also $y_x\in C([0,1];L^2(0,T)).$ Moreover, if $a\in L^\infty(\Omega)$ and $\|y_0\|_{L^2(\Omega)}$ is sufficiently small, then $y$ extends as a global solution. In other words, $T$ can be taken as $T'$.
  \item[ii)] Let $a\in H^1(\Omega)$ and let $y_0\in H^3(\Omega)$ satisfy the computability conditions $y_0(0)=y_0(1)=y_0'(1)=0.$  Then, the local solution in part (i) enjoys the extra regularity $y\in C([0,T];H^3(\Omega))\cap L^2(0,T;H^4(\Omega)).$
\end{enumerate}
\end{prop}
\begin{proof}

\textbf{Step 1 - Local wellposedness: } At first, we set a nonlinear operator $\Upsilon$ as follows:
 \begin{equation}\label{soloperator2}y=[\Upsilon z](t):= S(t)y_0+\int_0^tS(t-s)Fz(s)ds,\end{equation} where $Fz:= -\lambda z+ a(\cdot)z_x(0,\cdot)-(I-K)[\left(z+v\right)\left(z_{x} + v_x\right)]$ with $v=\Phi(z).$ Here, we consider $\Upsilon$ on a set given by $S_{T,r}:=\{z\in Y_T,\,\|z\|_{Y_T}\le r\},$ where $Y_T$ is as in \eqref{ourspace}.  The parameters $T,r>0$ will be determined later. $S_{T,r}$ is a complete metric subspace of $Y_T$ with respect to the metric induced by the norm of $Y_T$.  Since $v=\Phi z$, due to \eqref{l0} we have
\begin{equation}\label{vzrel2}
   \|v\|_{C([0,T];L^2(\Omega))}\le C_0\|z(t)\|_{C([0,T];L^2(\Omega))}.
\end{equation}  Similarly, using \eqref{l1} we deduce
 \begin{equation}\label{vzrel4}
   \|v\|_{L^2(0,T);H^1(\Omega))}\le  C_1\|z(t)\|_{L^2(0,T);L^2(\Omega))}.
 \end{equation} Finally,

  \begin{multline}\label{vzrel5}
   \|v_x(x)\|_{L^2(0,T)}^2=\int_0^T\left|\int_0^x\tilde{k}_x(x,y)z(y,t)dy\right|^2dt
   \le \left(\int_0^1|\tilde{k}_x(x,y)|^2dy\right)\|z\|_{L^2(0,T);L^2(\Omega))}^2\,,
 \end{multline}  from which it follows that

   \begin{equation}\label{vzrel6}
   \sup_{x\in (0,1)}\|v_x(x)\|_{L^2(0,T)}\\
   \le \|z\|_{L^2(0,T);L^2(\Omega))}\sup_{x\in (0,1)}\left(\int_0^1|\tilde{k}_x(x,y)|^2dy\right)^\frac{1}{2}.
 \end{equation}
Combining \eqref{vzrel2}, \eqref{vzrel4}, and \eqref{vzrel6}, we have
   \begin{equation}\label{vzrel7}
   \|v\|_{Y_T}\le c_{\tilde{k}}\|z\|_{Y_T}\,,
 \end{equation} where $c_{\tilde{k}}>0$ is a constant which only depends on various finite norms of $\tilde{k}$.  Taking the $Y_{T}$ norm of both sides of \eqref{soloperator2}, using the same semigroup estimates on $Y_T$ and the boundedness of $I-K$, we obtain
    \begin{multline}\label{vzrel8}
   \|\Upsilon z\|_{Y_T}\le c_0\|y_0\|_{Y_T}+c_1\int_0^{T}\left\|[Fz](\cdot,s)\right\|_{L^2(\Omega)}ds\\
   \le c_0\|y_0\|_{Y_T}+c_1\int_0^{T}\left\|a(\cdot)z_x(0,s)-\lambda z-(I-K)[(z+v)(z_x+v_x)]\,\right \|_{L^2(\Omega)}ds\\
   \le c_0\|y_0\|_{Y_T}+c_1\int_0^{T}\left[\left\|a(\cdot)z_x(0,s)-\lambda z\right\|_{L^2(\Omega)}+\left\|zz_x\right \|_{L^2(\Omega)}+\left\|zv_x\right \|_{L^2(\Omega)}+\left\|vz_x\right \|_{L^2(\Omega)}+\left\|vv_x\right \|_{L^2(\Omega)}\right]ds\\
   \le c_0\|y_0\|_{Y_T}
   +c_1\left[(1+T)^\frac{1}{2}\sqrt{T}(1+\|a\|_{L^2(\Omega)})\|z\|_{Y_T}+(T^\frac{1}{2}+T^\frac{1}{3})
   \left(\|z\|_{Y_T}^2+2\|z\|_{Y_T}\|v\|_{Y_T}+\|v\|_{_{Y_T}}^2\right)\right]\\
   \le c_0\|y_0\|_{Y_T}+c_1\left[(1+T)^\frac{1}{2}\sqrt{T}(1+\|a\|_{L^2(\Omega)})+(1+3c_{\tilde{k}})(T^\frac{1}{2}+T^\frac{1}{3})
   \|z\|_{Y_T}\right]\|z\|_{Y_T},
 \end{multline} where the fourth inequality follows from \cite[Lemma 3.1]{BSZ03}. Let us set $r=2c_0\|y_0\|_{Y_T}$\,, and choose $T>0$ to be small enough that $$c_1\left[(1+T)^\frac{1}{2}\sqrt{T}(1+\|a\|_{L^2(\Omega)})+(1+3c_{\tilde{k}})(T^\frac{1}{2}+T^\frac{1}{3})
   r\right]\le\frac{1}{2}.$$ With such choice of $(r,T)$, we get $\|\Upsilon z\|_{Y_T}\le r$ for all $z\in S_{T,r}$.  Therefore, $\Upsilon$ is a map from $S_{T,r}$ into $S_{T,r}$.

   Now, we claim that $\Upsilon$ is indeed a contraction on $S_{T,r}$ if $T$ is sufficiently small.  In order to see this, let $z,z'\in S_{T,r}.$  Then, similar to \eqref{vzrel8}, we have
   \begin{multline}\label{vzrel9}
   \|\Upsilon z-\Upsilon z'\|_{Y_T}\le c_1\int_0^{T}\left\|[Fz-Fz'](\cdot,s)\right\|_{L^2(\Omega)}ds\\
   \le c_1\int_0^{T}\left\|a(\cdot)(z_x(0,s)-z'_x(0,s))-\lambda (z-z')\right\|_{L^2(\Omega)}ds\\
   +c_1\int_0^T\left[\left\|zz_x-z'z'_x\right \|_{L^2(\Omega)}+\left\|zv_x-z'v'_x\right \|_{L^2(\Omega)}+\left\|vz_x-v'z'_x\right \|_{L^2(\Omega)}+\left\|vv_x-v'v'_x\right \|_{L^2(\Omega)}\right]ds\\
   \le c_1(1+T)^\frac{1}{2}\sqrt{T}(1+\|a\|_{L^2(\Omega)})\|z-z'\|_{Y_T}+c_1(T^\frac{1}{2}+T^\frac{1}{3})(\|z\|_{Y_T}+\|z'\|_T)\|z-z'\|_{Y_T}\\
   +c_1(T^\frac{1}{2}+T^\frac{1}{3})(\|z'\|_{Y_T}\|v-v'\|_{Y_T}+\|v\|_T\|z-z'\|_{Y_T})+c_1(T^\frac{1}{2}+T^\frac{1}{3})(\|z\|_{Y_T}\|v-v'\|_{Y_T}+\|v\|_{Y_T}\|z-z'\|_{Y_T})\\
   +c_1(T^\frac{1}{2}+T^\frac{1}{3})(\|v\|_{Y_T}+\|v'\|_{Y_T})\|v-v'\|_{Y_T}.
 \end{multline}

 Now, using \eqref{vzrel7}, for the same $r$ as before, but choosing $T$ smaller if necessary, we obtain
 $$\|\Upsilon z-\Upsilon z'\|_{Y_T}\le\rho \|z-z'\|_{Y_T}$$ for some $\rho\in (0,1)$.  Then, by the Banach contraction theorem, we get the existence and uniqueness of a local solution in $S_{T,r}$.

\textbf{Step 2 - Regularity:} Let $y_0\in D(A)$. We define the closed space $$B_{T,r}:=\{(\psi,\varphi)\in Y_T^3\times Y_T\,|\,\psi\in C([0,T];L^2(\Omega))\cap L^2(0,T;H^1(\Omega)), \varphi=\psi_t, \|\psi\|_{Y_T^3}+\|\varphi\|_{Y_T}\le r\}.$$
Now, given $(z,\tilde{z})\in B_{T,r}$ let $q$ be a solution of

\begin{equation}\label{nonlinKdV-wp-q-2}
 	\begin{cases}
 	\displaystyle q_{t} +q_{x} + q_{xxx} + {\lambda} q \\
 = a(x)q_x(0,\cdot) - (I-K)[(\tilde{z}+\tilde{v})(z_x+v_{x})-(z+v)(\tilde{z}_{x}+\tilde{v}_{x})] & \text { in } \Omega\times (0,T),\\
    {q}(0,t) = {q}(1,t) = {q}_{x}(1,t) = 0 & \text { in } (0,T),\\
    {q}(x,0)= q_0:=-y_0'-y_0'''-\lambda y_0+a(x)y_0'(0)-(y_0+v_0)(y_0'+v_0') & \text { in } \Omega,
 	\end{cases}
\end{equation} where $v=\Phi(z)$, $v_0=\Phi(y_0)$, $\tilde{v}=\Phi(\tilde{z})$. Set $y=y_0+\int_0^tqds$.  Then, $y_t=q$ and $y$ solves

\begin{equation}\label{nonlinKdV-wp-q-1}
 	\begin{cases}
 	\displaystyle y_{t} +y_{x} + y_{xxx} + {\lambda} y = a(x)y_x(0,\cdot) - (I-K)[(z+v)(z_x+v_x)] & \text { in } \Omega\times (0,T),\\
    {y}(0,t) = {y}(1,t) = {y}_{x}(1,t) = 0 & \text { in } (0,T),\\
    {y}(x,0)=y_0 & \text { in } \Omega.
 	\end{cases}
\end{equation}
We set an operator $\Theta: (z,\tilde{z})\mapsto (y,q)$ associated with the system of equations given by \eqref{nonlinKdV-wp-q-2}-\eqref{nonlinKdV-wp-q-1}.   One can show that for suitable $r$ and small $T$, the operator $\Theta$ maps $B_{T,r}$ onto itself in a contractive manner.  This can be done by obtaining the same type of estimates given in Step 1 for both the solution of \eqref{nonlinKdV-wp-q-2} and \eqref{nonlinKdV-wp-q-1}.  Therefore, it has a unique fixed point whose first component is the regular local solution we are looking for.

\textbf{Step 3 - Global solutions:} Global wellposedness in $Y_T$ with small initial datum follows directly from the stabilization estimate proved in Section \ref{ProofofThm2}. 
\end{proof}
Global well-posedness of the nonlinear modified target system \eqref{ch414} now follows from the Proposition \ref{wellposednessprop2} that we just proved.

\section{Using a single controller}
\subsection{Smaller decay rate}
Although we studied the model \eqref{KdVBurgers} with two controls at the right hand side, it is also possible to use only one control. For example \cite{Cor14} proves exponential stability with the control acting only from the Neumann boundary condition when $L$ is not of critical length. When $L$ is not restricted to uncritical lengths, we can still obtain exponential stability with a single Dirichlet control rather than a Neumann control by using the pseudo-backstepping method above. However, this causes a smaller rate of decay. Consider for instance the plant
\begin{equation}\label{singlecontrol}
 	\begin{cases}
 	\displaystyle u_{t} + u_{x} + u_{xxx} =0 & \text { in } \Omega\times \mathbb{R_+},\\
    u(0,t) = 0, u(L,t) = U(t), u_{x}(L,t)=0  & \text { in } \mathbb{R_+},\\
    u(x,0)=u_0(x) & \text { in } \Omega.
 	\end{cases}
\end{equation}
Then the backstepping transformation \eqref{mod-transform} gives the following target system
\begin{equation}\label{HomKdVBurgers-single-1}
 	\begin{cases}
 	\displaystyle \tilde{w}_{t} +\tilde{ w}_{x} + \tilde{w}_{xxx} + \lambda \tilde{w} = \tilde{k}_y(x,0)\tilde{w}_x(0,t) & \text { in } \Omega\times \mathbb{R_+},\\
    \tilde{w}(0,t) = \tilde{w}(L,t) = 0, \tilde{w}_{x}(L,t) = -\int_0^L\tilde{k}_x(L,y)u(y,t)dy  & \text { in } \mathbb{R_+},\\
    \tilde{w}(x,0)=\tilde{w}_0(x):= u_0-\int_0^x\tilde{k}(x,y)u_0(y)dy & \text { in } \Omega.
 	\end{cases}
\end{equation}
If we multiply the above system by $\tilde{w}$, integrate over $(0,L)$, and use integration by parts, the Cauchy-Schwarz inequality, and boundary conditions we obtain
\begin{equation}\label{HomKdVBurgers-single-2}
 	\frac{1}{2}\frac{d}{dt}\|\tilde{w}(t)\|_{L^2(\Omega)}^2+ \left(\lambda-\frac{1}{2}\|\tilde{k}_y(\cdot,0)\|_{L^2(\Omega)}^2-\frac{1}{2}\|\tilde{k}_x(L,\cdot )\|_{L^2(\Omega)}^2\|(I-K)^{-1}\|_{B[L^2(\Omega)]}^2 \right)\|\tilde{w}(t)\|_{L^2(\Omega)}^2\le 0.
\end{equation}
Comparing \eqref{HomKdVBurgers-single-2} and \eqref{HomKdVBurgers-3} we see that we still achieve \eqref{HomKdVBurgers-4} where $\alpha= \lambda-\frac{1}{2}\|\tilde{k}_y(\cdot,0)\|_{L^2(\Omega)}^2$ is replaced by $\beta=\left(\lambda-\frac{1}{2}\|\tilde{k}_y(\cdot,0)\|_{L^2(\Omega)}^2-\frac{1}{2}\|\tilde{k}_x(L,\cdot )\|_{L^2(\Omega)}^2\|(I-K)^{-1}\|_{B[L^2(\Omega)]}^2 \right)$. Recall that in Lemma \ref{alemlambda} we showed $\|\tilde{k}_y(\cdot,0)\|_{L^2(\Omega)} \sim \lambda $. One can also get $\|\tilde{k}_x(L,\cdot )\|_{L^2(\Omega)} \sim \lambda  $ by using similar arguments. Moreover, using the calculations in \cite{Liu03} we deduce that $\|(I-K)^{-1}\|_{B[L^2(\Omega)]}\sim 1+\lambda e^{C\lambda}$ where $C>0$ depends only on $L$. Hence positivity of $\beta$ is guaranteed for a sufficiently small choice of $\lambda$. As a result, the decay rate decreases while the exponential stability still holds.
\begin{rem} Glass and Guerrero \cite{Glass10}  proved that \eqref{singlecontrol} is exactly controllable if and only if $L$ does not belong to a set of critical lengths $\mathcal{O}$, defined by them, which is different than $\mathcal{N}$.  They showed that if $L\in \mathcal{O}$, then the following problem has a nontrivial solution:
\begin{equation}\label{phieq1}
 	\begin{cases}
 	\varphi'''+\varphi'=\lambda \varphi & \text { in } (0,L),\\
    \varphi(0) = \varphi(L) = \varphi'(0)=\varphi''(L)=0.
 	\end{cases}
\end{equation} Moreover, it was found that for any $u_0\in L^2(\Omega)$ and control input $U\in L^2(0,T)$, the function 
\begin{equation}\label{ftcons}
e^{-\lambda t}\int_0^Lu(x,t;U)\varphi(x)dx 
\end{equation} remains constant in time, where $u=u(x,t;U)$ is the corresponding trajectory for \eqref{singlecontrol} with control input $U$. Regarding the same KdV system, we construct a boundary feedback which yields stabilizability with a certain decay rate for all $L>0$ including $L\in \mathcal{O}$. This shows that for a certain control input $U$, the integral $\int_0^Lu(x,t;U)\varphi(x)dx$ decays in time because $\left|\int_0^Lu(x,t;U)\varphi(x)dx\right|\le \|u(t;U)\|_{L^2(\Omega)}\|\varphi\|_{L^2(\Omega)}$,  where $\|u(t;U)\|_{L^2(\Omega)}$ decays exponentially. Therefore, since \eqref{ftcons} is invariant in time, it follows that $|e^{-\lambda t}|$ must increase, which is only possible if $Re(\lambda)<0$. However, if $Re(\lambda)<0$, then $|e^{-\lambda t}|\rightarrow \infty$. Since \eqref{ftcons} is valid for all control inputs, we conclude in particular that $\int_0^Lu(x,t;0)\varphi(x)dx\rightarrow 0$ $(U(t)\equiv 0$). This shows that any uncontrollable trajectory corresponding to some initial state $u_0$ with no feedback is attracted to the orthogonal complement of the span of the set consisting of the real and imaginary parts of the nontrivial solutions of \eqref{phieq1}.
\end{rem}
\begin{rem}
Regarding the nonlinear KdV equation with same boundary conditions, Glass and Guerrero \cite{Glass10} obtained the exact controllability given that the domain is of uncritical length and the initial and final states are small.  This was not an if and only if statement unlike the linear problem.  Therefore, we do not know whether the exact controllability on domains of critical lengths is true with the control acting at the right Dirichlet b.c.  On the other hand, our feedback control design can be easily extended to the nonlinear system for small data as in Section 2.2.   The nonlinear case is quite interesting because maybe the nonlinear term $uu_x$ is creating a further stability effect on the solutions, which might drive them to zero by themselves. As far as we know, the exact controllability for the (nonlinear) KdV equation as well as the decay of solutions to zero by themselves in the presence of the nonlinear term $uu_x$ remain as open problems on critical length domains, see for instance Cerpa \cite{cer14}.
\end{rem}

\subsection{A second order feedback law}
 In this section, we check whether it is possible to stabilize the solutions of the KdV equation by using the feedback law $u(L,t)=U(t)$ with the input $U(t)=u_{xx}(L,t)$.  In order to gain some intuition regarding this problem, let us consider the linearized model \eqref{singlecontrol} with the input $U(t)=u_{xx}(L,t)$.  Multiplying the main equation in \eqref{singlecontrol} by $u$ and integrating over $\Omega$ by using the given boundary conditions, one obtains
\begin{equation}\label{newL2iden}
  \frac{d}{dt}\|u(t)\|_{L^2(\Omega)}^2 = -\frac{3}{2}|u(L,t)|^2-\frac{1}{2}|u_x(0,t)|^2\le 0.
\end{equation} The inequality \eqref{newL2iden} shows that $\|u(t)\|_{L^2(\Omega)}$ is non-increasing, but it is not clear whether it decays to zero.  In order to better understand the behavior of the solution, one generally must study the spectral properties of the corresponding evolution.  Regarding \eqref{singlecontrol}, one can study the operator $$Au=-u'''-u',\,\,\,D(A)=\{u\in H^3(\Omega)\,|\,u(0)=u'(L)=u(L)-u''(L)=0\}.$$  However, it is quite difficult to analyze the eigenvalues of this operator. This is because the characteristic equation corresponding to the eigenvalue problem $Au=\lambda u$ takes the form $r^3+r+\lambda =0$, which is not easy to study. Therefore, we will consider this problem on the rather simplified model given below, neglecting the first order term $u_x$:
\begin{equation}\label{simplemodel}
 	\begin{cases}
 	\displaystyle u_{t} + u_{xxx} =0 & \text { in } \Omega\times \mathbb{R_+},\\
    u(0,t) = 0, u(L,t) = u_{xx}(L,t), u_{x}(L,t)=0  & \text { in } \mathbb{R_+},\\
    u(x,0)=u_0(x) & \text { in } \Omega,
 	\end{cases}
\end{equation} where the inequality \eqref{newL2iden} takes the form
\begin{equation}\label{newL2iden2}
  \frac{d}{dt}\|u(t)\|_{L^2(\Omega)}^2 = -|u(L,t)|^2-\frac{1}{2}|u_x(0,t)|^2\le 0.
\end{equation}
\subsection*{Spectral properties}
 The operator which generates the evolution corresponding to  \eqref{simplemodel} is a third order dissipative differential operator given by $$Au=-u''',\,\,\,D(A)=\{u\in H^3(\Omega)\,|\,u(0)=u'(L)=u(L)-u''(L)=0\},$$ which has compact resolvent and spectrum involving countably many eigenvalues $\{\lambda_k\}_{k\in \mathbb{Z}}$ satisfying $\text{Re}\lambda_k\le 0$. Moreover, these eigenvalues satisfy the properties given in the lemma below, whose proof uses the approach presented in \cite[Prop 2.2]{Zhang001} and \cite[Prop 3.1]{Zhang002}.
\begin{lem}
  $\lambda_k = -\frac{8\pi^3}{3\sqrt{3}L^3}|k|^3 $ as $|k|\rightarrow \infty$. Moreover, $\exists \eta<0$ s.t. $\text{Re}\lambda_k<\eta$ $\forall k\in \mathbb{Z}$.
\end{lem}
\begin{proof}
  Let $\lambda$ be an eigenvalue of $A$. Then, $\text{Re}\lambda \le 0$, and  we can assume wlog that $\text{Im}\lambda\le 0$ since $\bar{\lambda}$ is also an eigenvalue of $A$. Let us first see that $\text{Re}\lambda$ cannot be equal to zero.  To this end, let $i\xi$ be an eigenvalue with $\xi\in \mathbb{R}$ and $u$ be the corresponding eigenvector.  Then, we note that $$\text{Re}(Au,u)_{L^2(\Omega)}=-|u|^2(L)-\frac{1}{2}|u'(0)|^2=\text{Re}(i\xi\|u\|_{L^2(\Omega)}^2)=0.$$ We get $u(L)=u'(0)=0$, which implies  together with other boundary conditions $u\equiv 0$. This contradicts the fact that $u$ was an eigenvector.  Hence, $\text{Re}\lambda<0$.

  Let $r_i$ $(i=0,1,2)$ be the three roots of the characteristic equation $r^3+\lambda=0$ corresponding to the ode \begin{equation}\label{Aulambdau}
                                                                                                                 Au=\lambda u,
                                                                                                               \end{equation}with $r_1$ being the root in the first quadrant. Note that we have $r_i=\alpha^i r_1$, $\alpha=e^{\frac{2\pi i}{3}}$.
The solution of \eqref{Aulambdau} is $u(x)=\sum_{i=0}^2c_ie^{r_i x}$ where $\{c_i, i=0,1,2\}$, due to given boundary conditions, satisfy the system of equations
\[
\begin{cases} \sum_{i=0}^2c_i =0 \\ \sum_{i=0}^2c_ir_ie^{r_i L} =0 \\ \sum_{i=0}^2 c_i(1-r_i^2e^{r_i L})=0, \end{cases}\] which has a nontrivial solution if
$$\sum_{i=0}^2a_{ij}r_i(1-r_j^2)e^{(r_i+r_j)L}=0,$$ where $a_{12}=a_{20}=a_{01}=1$ and $a_{21}=a_{02}=a_{10}=-1$. Multiplying the above equation by $e^{-r_0L}$ and neglecting the relatively small terms which involve $e^{(r_1+r_2-r_0)L}$, we get
$$(1+\alpha)(1+\alpha^2r_0^2)e^{r_2L}+(1+\alpha r_0^2)e^{r_1L}=0.$$  Now, we see that we can further neglect the terms $(1+\alpha)e^{r_2L}$ and $e^{r_1L}$ since they are much smaller than $(1+\alpha)\alpha^2r_0^2e^{r_2L}$ and $\alpha r_0^2e^{r_1L}$, respectively.  Therefore, asymptotically, we get
$$(1+\alpha)\alpha e^{r_2L}+e^{r_1L}=0.$$ Observe that $\alpha(1+\alpha)=-1$.  Therefore, the asymptotic relation is reduced to $$e^{(r_2-r_1)L}=1,$$
from which, together with the definition of $\alpha,r_1$, and $r_2$, it follows that $\lambda_k=-r_{0,k}^3=-\frac{8\pi^3}{3\sqrt{3}L^3}|k|^3$ asymptotically. This property combined with the fact that $\text{Re}\lambda_k<0$ proves the second part of the lemma.
\end{proof}
It is easy to see that $A^*$ is defined by
$$A^*w=w''',\,\,\,D(A^*)=\{w\in H^3(\Omega)\,|\,w(0)=w'(0)=w(L)+w''(L)=0\}.$$
Now, the following result follows classically from the spectral properties of $A$ above.
\begin{lem}(\cite[Prop 2.1, 2.2]{Zhang001}, \cite[Prop 3.2]{Zhang002})
  $A$ is a discrete spectral operator, and all but a finite number of eigenvalues of $A$ correspond to one dimensional projections $E(\lambda;T)$. Both $A$ and $A^*$ have complete sets of eigenvectors, $\{\phi_k\}_{k\in \mathbb{Z}}$ and $\{\psi_j\}_{j\in \mathbb{Z}}$, respectively, satisfying $(\phi_k,\psi_j)_{L^2(\Omega)}=\delta_{kj}$ and forming dual Riesz bases for $L^2(\Omega)$.
\end{lem}
\subsection*{A special multiplier and stabilization} We set $Y:=\sum_{k}Y_k$ where $Y_k$ is defined by $$Y_k(u)=(u,\psi_k)_{L^2(\Omega)}\psi_k.$$ Then $Y$ is bounded and positive definite by the uniform $\ell^2$ convergence property of $\{\phi_k\}_{k\in \mathbb{Z}}$ and $\{\psi_j\}_{j\in \mathbb{Z}}$. Second, we define the symmetric, bounded, nonnegative operator $X=\sum_{k}\xi_kY_k$, where $\xi_k=-\frac{1}{2\text{Re}\lambda_k}$, which satisfies the additional property $$A^*X+XA+Y=0.$$ Using $Xu$ as a multiplier, we compute
\begin{equation}\label{XYcomp}
  \frac{d}{dt}(Xu,u)_{L^2(\Omega)} = (\frac{d}{dt}Xu,u)_{L^2(\Omega)}+(Xu,u_t)_{L^2(\Omega)}
  =(XAu,u)_{L^2(\Omega)}+(Xu,Au)_{L^2(\Omega)}
  =-(Yu,u)_{L^2(\Omega)}.
\end{equation}

\eqref{XYcomp} together with \eqref{newL2iden2} imply that
\begin{equation}\label{gronprep1}
  \frac{d}{dt}((I+X)u,u)_{L^2(\Omega)}\le -(Yu,u)_{L^2(\Omega)}.
\end{equation} Now, integrating in time and using the positive definiteness of $(I+X)$ and $Y$, applying Gronwall's inequality, we obtain the exponential decay of solutions in $L^2(\Omega)$, and the following theorem follows.
\begin{thm}\label{specthm}
  Let $u$ be a solution of the linearized KdV equation in \eqref{simplemodel}.  Then, there exists some $\gamma>0$ independent of $u_0$ such that $$\|u(t)\|_{L^2(\Omega)} \lesssim \|u_0\|_{L^2(\Omega)}e^{-\gamma t}$$ for $t\ge 0$.
\end{thm}
\begin{rem}
  Note that Theorem \ref{specthm} was proved for the simplified linearized model \eqref{simplemodel}.  The situation is more challenging for more general models involving other terms such as $u_x$ and/or the nonlinear term $uu_x$. Moreover, the approaches of \cite{Zhang001} and \cite{Zhang002} do not seem to directly apply to these more general problems under the boundary conditions $u(0)=u'(L)=u(L)-u''(L)=0$. This is due to the fact that the eigenvalue analysis gets much more challenging with a more complicated third order characteristic equation.  In order to simplify the eigenvalue analysis, one can still use the simpler operator $Au=-u'''$ treating $-u_x$ and/or $uu_x$ as source term(s).  But then, one needs the adjoint operator $A^*$ to satisfy very desirable boundary conditions so that the trace terms are cancelled out when one applies the special multiplier. However, this does not become the case with the given boundary conditions in the model. This issue is not present with the boundary conditions used in \cite{Zhang001} and \cite{Zhang002}.  Therefore, the case of more general equations with the first order term $u_x$ and/or the nonlinear term $uu_x$ remain interesting open problems.
\end{rem}
\section{Numerical simulations} We modify the finite difference scheme given in \cite{Pazato} to fit it into the present situation, where we have first order trace terms in the main equations of the target systems and inhomogeneous boundary inputs of feedback type in the original plant.   We numerically solve the KdV equation both in the controlled and uncontrolled cases.  We are also able to verify our main result numerically.  First, we simulate an uncontrolled solution of the KdV equation and then we simulate the controlled solution. From our simulations, one can see that the boundary controllers constructed using a pseudo-kernel effectively stabilize the solutions with a suitable choice of $\lambda$. The calculations are performed in Wolfram Mathematica\textsuperscript{\textregistered}11.

For simplicity, we consider only the linearised problem. The nonlinear problem can be treated in a similar way by including an additional fixed point argument to the algorithm we describe here.  We use the notation given in \cite{Pazato}. To this end, we set the discrete space $$X_J:=\{\tilde{w}=(\tilde{w}_0,\tilde{w}_1,...,\tilde{w}_J)\in \mathbb{R}^{J+1}\,|\,\tilde{w}_0=\tilde{w}_{J-1}=\tilde{w}_J=0\},$$ and the difference operators $\displaystyle (D^+\tilde{w})_j:=\frac{\tilde{w}_{j+1}-\tilde{w}_j}{\delta x}$, $\displaystyle (D^-\tilde{w})_j:=\frac{\tilde{w}_{j}-\tilde{w}_{j-1}}{\delta x}$ for $j=1,...,J-1$, and $\displaystyle D=\frac{1}{2}(D^++D^-)$. We will call the space and time steps $\delta x$ and $\delta t$ for $j=0,...,J,$ and $n=0,1,...,N$, respectively. Using this notation, the numerical approximation of the linearised target system \eqref{HomKdVBurgers-1} takes the form
\begin{eqnarray}
  \label{wjn1}\frac{\tilde{w}_{j}^{n+1}-\tilde{w}_j^n}{\delta t}+(\mathcal{A}\tilde{w}^{n+1})_j+\lambda \tilde{w}_j^{n+1}&=& \tilde{k}_y(x_j,0)\frac{\tilde{w}_{1}^{n}}{\delta x},\hspace{.1in} j=1,...,J-1\\
  \label{wjn2}\tilde{w}_0=\tilde{w}_{J-1}=\tilde{w}_J &=& 0, \\
  \label{wjn3} \tilde{w}_0  &=&\int_{x_{j-\frac{1}{2}}}^{x_{j^+\frac{1}{2}}}\tilde{w}_0(x)dx,\hspace{.1in} j=1,...,J-1,
\end{eqnarray} where $x_{j\mp\frac{1}{2}}=(j\mp\frac{1}{2})\delta x$, $x_j=j\delta x$. The $(J-1)\times (J-1)$ matrix $\mathcal{A}$ approximates $\tilde{w}_x+\tilde{w}_{xxx}$ and is defined by $\mathcal{A}:=D^+D^+D^-+D$. Let us set $\mathcal{C}:=(1+\delta t\lambda)I+\delta t A$.  Then, from the main equation, we obtain
$\tilde{w}_{j}^{n+1}=\mathcal{C}^{-1}\left(\tilde{w}_j^n+\frac{\delta t}{\delta x}\tilde{k}_y(x_j,0)\tilde{w}_{1}^{n}\right)$ for $j=1,...,J-1$.

In order to approximate the solution of the original plant \eqref{KdVBurgers} with feedback controllers, we use the succession idea in the proof of Lemma \ref{inverselem}.  Note that given $\tilde{w}$, $v$ is the fixed point of the equation $v=K(\tilde{w}+v)$.  For numerical purposes, let $m$ denote the number of iterations in the succession and set $v^0=\mathcal{K}\tilde{w}$, $v^{k}:=\mathcal{K}(\tilde{w}+v^{k-1})$ for $1\le k\le m$, where $\mathcal{K}$ is the numerical approximation of the integral in the definition of $K$.  Then, $v^m$ is an approximation of $v=\Phi(\tilde{w})$, and one gets an approximation of the original plant by setting $u(x_j,t_n):=\tilde{w}(x_j,t_n)+v^m(x_j,t_n)$.

On a domain of critical length, one can find time-independent solutions, as we mentioned in the introduction.  Figure \ref{uncont-sol} below shows such a solution on a domain of length $L=2\pi$ whose $L^2$-norm is preserved in time.
\text{}\\
\begin{figure}[H]
  \centering
   \includegraphics[scale=0.75]{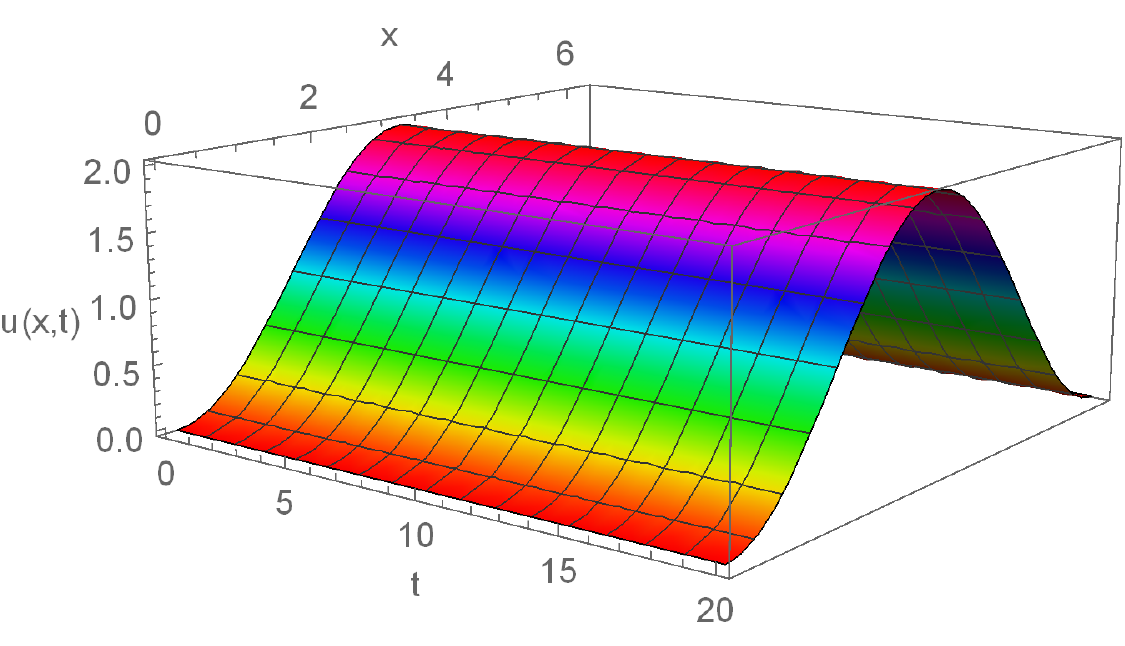}
  \caption{Uncontrolled solution with initial datum $u_0=1-\cos(x)$ on a domain of length $2\pi$.}\label{uncont-sol}
\end{figure}

If one applies the boundary controllers constructed with the same initial profile that the uncontrolled solution has in Figure \ref{uncont-sol}, then the new solution will decay to zero as we illustrate in Figure \ref{controlled}.

\begin{figure}[H]
  \centering
   \includegraphics[scale=0.75]{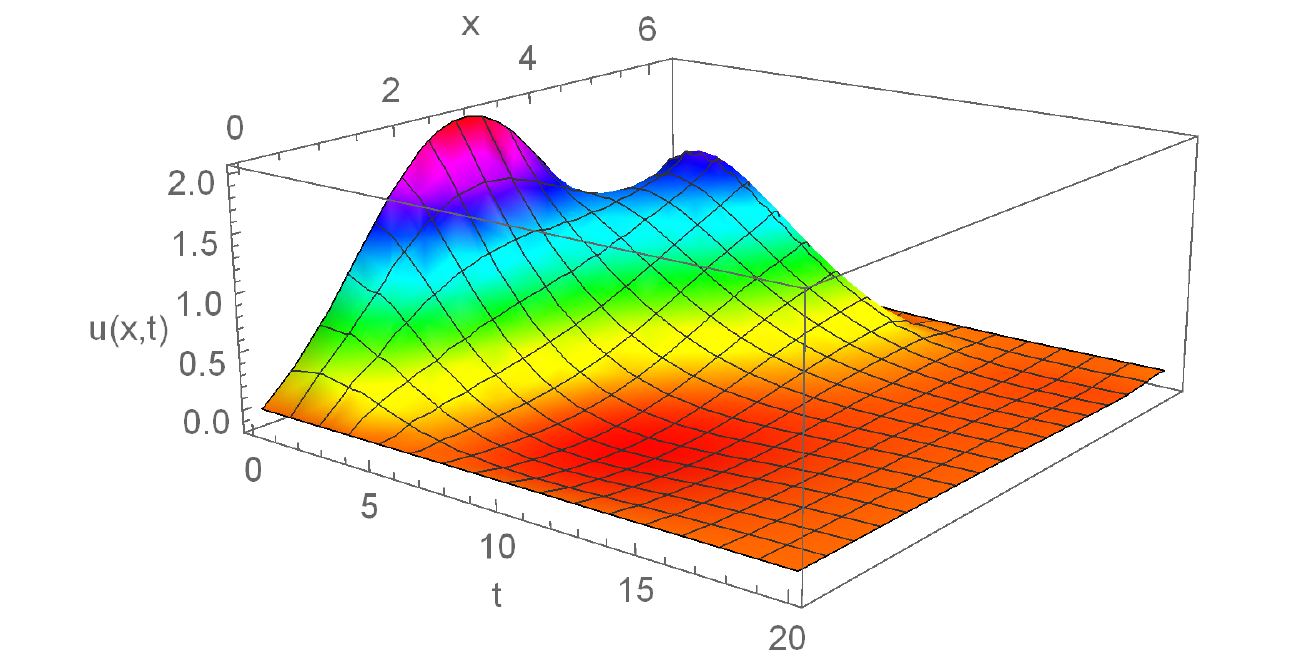}
  \caption{Controlled solution with initial datum $u_0=1-\cos(x)$,  $\lambda=0.03$, with a controller using the pseudo-kernel $\tilde{k}$ on a domain of length $2\pi$.}\label{controlled}
\end{figure}

Figure \ref{u1t} shows the controller behavior on the Dirichlet boundary condition at the right endpoint. As one can see, less control is needed as the wave gets supressed.\\

\begin{figure}[H]
  \centering
   \includegraphics[scale=0.75]{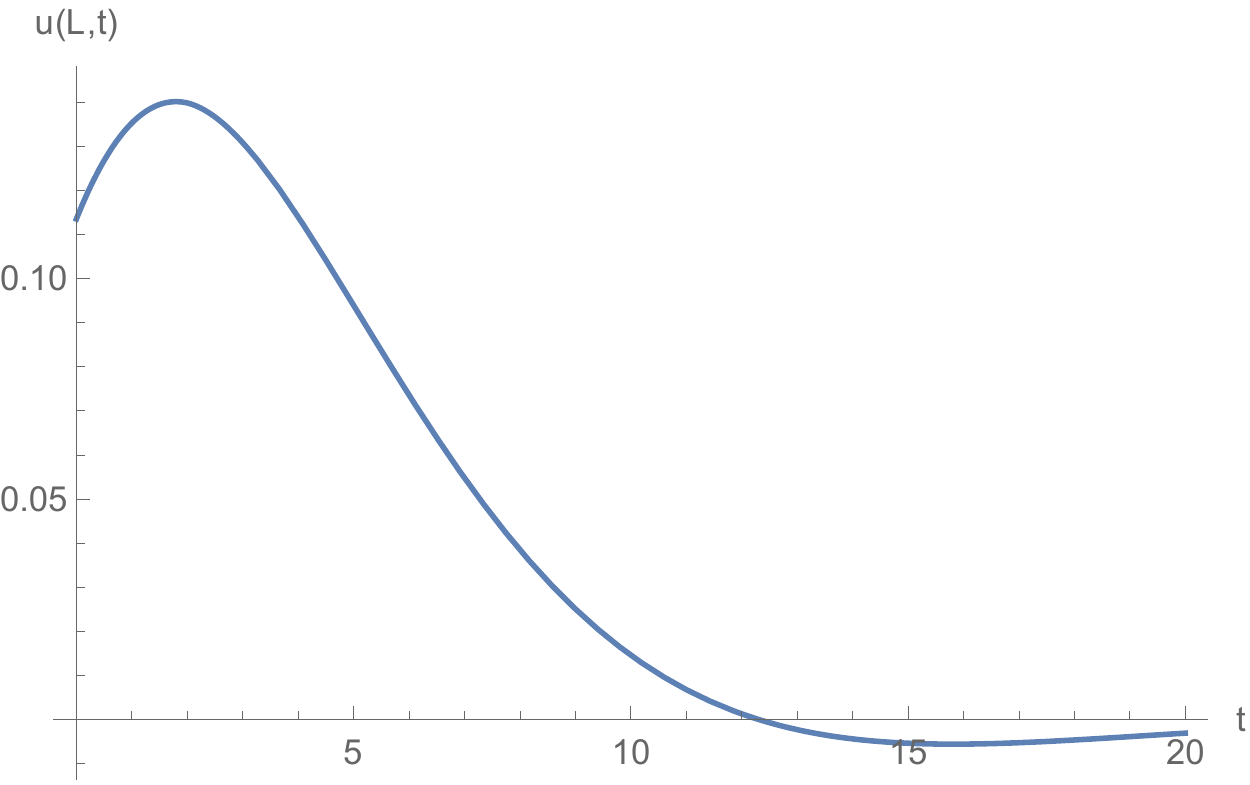}\\
  \caption{Dirichlet controller at the right endpoint ($\lambda=0.03$)}\label{u1t}
\end{figure}

\section*{Acknowledgments} We would like to express our gratitude to the anonymous referee whose valuable insights significantly improved the quality of this article. We would also like to thank Katherine H. Willcox from Izmir Institute of Technology for her English editing of this paper.

\bibliographystyle{siam}
\bibliography{myreferences}

\end{document}